\documentclass[10pt]{article}
\usepackage[T1]{fontenc}
\usepackage{amsmath,amsfonts,amsthm,mathrsfs,amssymb,bm}
\usepackage{cite}
\usepackage{multirow}
\usepackage{graphicx,float}
\usepackage{subfigure}
\usepackage{placeins}
\usepackage{color}
\usepackage{indentfirst}
\usepackage[bookmarks=true]{hyperref}
\usepackage{bookmark}
\allowdisplaybreaks
\numberwithin{equation}{section}
\newcommand{\R}{{\mathbb R}}

\newcommand{\E}{{\bm E}}

\newcommand{\be}{\begin{eqnarray}}
	\newcommand{\ben}{\begin{eqnarray*}}
		\newcommand{\en}{\end{eqnarray}}
	\newcommand{\enn}{\end{eqnarray*}}
\newtheorem{theorem}{Theorem}[section]
\newtheorem{lemma}[theorem]{Lemma}

\newtheorem{definition}[theorem]{Definition}
\newtheorem{remark}[theorem]{Remark}

\definecolor{Tao}{rgb}{1.000,0.000,0.000}
\definecolor{LJW}{rgb}{0.000,0.000,1.000}

\begin{document}
\renewcommand{\theequation}{\arabic{section}.\arabic{equation}}
\begin{titlepage}
	\title{On the direct and inverse electromagnetic scattering in a parallel-plate waveguide}
	
	\author{
		Jiawei Liang\thanks{School of Mathematical Sciences, University of Electronic Science and Technology of China, Chengdu, Sichuan 611731, China. Email:{\tt 202411110414@std.uestc.edu.cn}},\;
		Maojun Li\thanks{School of Mathematical Sciences, University of Electronic Science and Technology of China, Chengdu, Sichuan 611731, China. Email:{\tt limj@uestc.edu.cn}},\;
		Tao Yin\thanks{State Key Laboratory of Mathematical Sciences and Institute of Computational Mathematics and Scientific/Engineering Computing, Academy of Mathematics and Systems Science, Chinese Academy of Sciences, Beijing 100190, China. Email:{\tt yintao@lsec.cc.ac.cn}}
	}
\end{titlepage}
\maketitle

\begin{abstract}
This paper devotes to providing rigorous theoretical analysis of the wellposedness of the direct problem and the uniqueness of the inverse problem of electromagnetic scattering in a parallel-plate waveguide. The direct problem is reduced to an equivalent boundary value problem on a bounded domain by introducing an exact transparent boundary condition in terms of the electric-to-magnetic Calder\'on operator which can be explicitly represented as a series expansion. Then the wellopsedness of the reduced problem in appropriate Sobolev spaces is proved via the variational approach provided by some necessary properties of the Calder\'on operator and Helmholtz decomposition. Relying on the Green's representation formula and a reciprocity relation, the probe method, finally, is utilized to show the uniqueness of the inverse obstacle problem.
\end{abstract}

\section{Introduction}
Scattering problems in waveguides arise from diverse scientific and engineering fields, such as underwater acoustics, integrated optics, geophysics, ultrasonic nondestructive testing and antenna systems. A great number of mathematical and numerical results are available with regard to its direct and inverse problems, especially for acoustics~\cite{MR2402567,MR2796088,MR1043753,MR2069154,MR4000070,MR4027036,MR4216871,MR4573256}, electromagnetics~\cite{MR4753915,MR4641642,MR4259674,MR3975368}, elastodynamics~\cite{MR2644191,MR4606484,MR4052749,MR2785835} as well as multiphysical problems \cite{MR2065370,MR4674007}. 

The scattering problems in waveguides are imposed in an unbounded domain which is located between two parallel plates or in a closed tube. It is well known that the studies on wave propagation phenomenon in waveguides are much more challenging than in free space. Specifically, because of the presence of the boundaries of waveguides, only a finite number of modes can propagate over long distances, which are called propagating modes, while the other modes named evanescent modes will decay exponentially \cite{MR1043753,MR3975368,MR1387405}. Additionally, related to the type of boundary conditions, another difficulty arising in the analysis of the waveguide scattering problems is that the possible existence of trapped modes which can destroy the uniqueness of the solution \cite{MR4052749,MR1121393,MR1265871,MR2158985,MR2144620,MR2284257,MR3439069}. In general, the scattering problems can only be shown to be wellposed in the Fredholm sense except for a countable set of frequencies, which might correspond in the part to the trapped modes. 

This work is interested in the wave scattering problems in a parallel-plate waveguide~\cite{MR4052749,MR1387405} and some works have been carried out for the acoustic and elastic cases. The problem of acoustic wave scattering by a Dirichlet obstacle in a parallel-plate waveguide is shown in \cite{MR2402567,MR1625273} to be uniquely solvable except possibly for an infinite series of exceptional values of the wavenumber with no finite accumulation point and some geometric conditions are given to ensure the
uniqueness for all wavenumbers. The corresponding analysis has been extended to the medium scattering problems \cite{MR2796088} in a planar waveguide.
For the elastic scattering problem, extending these techniques are not straightforward, mainly due to the additional non-standard properties of the elastic modes, called the Lamb modes \cite{MR2644191}. A new expansion formulation is introduced in \cite{fraser1976orthogonality} to complete the family of Lamb modes, under which the wellposedness can be established for time-harmonic elastic problems in two- and three-dimensional parallel-plate waveguides \cite{MR4606484}. However, the related wellposedness analysis for the electromagnetic wave scattering problems in a parallel-plate waveguide is still unclear.

As to the inverse scattering, compared with 
the problems in free space, fewer methods are available to establish the uniqueness. Since a part of information is only detectable in propagation modes far away from the scatterers \cite{MR1387405}, the uniqueness can generally be constructed in the sense of near-field imaging \cite{MR1809490,MR2069154,MR3975368}. For the acoustic scattering in a parallel-plate waveguide, the uniqueness for the identification of inhomogeneity has been shown by means of the Dirichlet-to-Neumann map~\cite{MR1809490} or the outgoing Green operator~\cite{MR2069154}. Alternatively, the probe method \cite{MR1853728} is another effective way for the uniqueness analysis and its applications the scattering problems in tubular waveguides for elastic waves \cite{MR2785835} and electromagnetic waves \cite{MR3975368} has been considered.

The aim of this paper is to establish both the wellposdeness of the direct scattering problem and the uniqueness of the inverse obstacle problem for the electromagnetic scattering by an impedance obstacle in a PEC parallel-plate waveguide. This problem is also related to the reduced model studied in~\cite{MR4641642} for the signal-propagation problems in axon, however, only the wellposedness in the case of TE/TM is given in that work. By imposing an appropriate radiation condition according to the mode expansion, uniqueness of original model problem is proved. An exact transparent boundary condition (TBC) is introduced to truncate the unbounded domain by defining an electro-to-magnetic Calder\'on operator. This Calder\'on operator can be expressed as a Fourier series which is shown to be well-defined. Then relying on some newly derived properties of the Calder\'on operator, Helmholtz decomposition and compactness, the wellposedness of the corresponding variational formulation on the truncated cylindrical domain is deduced. Concerning the inverse obstacle problem with near-field measurement due to point source, the uniqueness result is established via the probe method.

The rest of this paper is organized as follows. In Section \ref{sec:2}, the model problem and its uniqueness is presented. The wellposedness analysis for the direct problem, which is truncated by means of an exact TBC, is established in Section \ref{sec 3} through the variational approach. With the help of a reciprocity relation, the uniqueness of the inverse obstacle problem is deduced by the probe method in Section \ref{sec 4}. Final conclusions are given in Section \ref{sec 5}.

\section{Preliminaries}
\label{sec:2}

\subsection{Model problem}
\label{sec:2.1}

Let $D$ be a bounded Lipschitz domain included in a parallel-plate waveguide $\R^3_Z=\{\bm x=(x_1,x_2,x_3)^\top\in\R^3: x_3\in(0,Z)\}$ whose upper and lower boundaries are denoted by $\Gamma^+$ and $\Gamma^-$, respectively. Denote $\Omega=\R^3_Z\backslash\overline{D}$. Consider the following electromagnetic scattering problem in $\Omega$ as
\begin{equation}\label{model:E}
	\left\{
	\begin{aligned}
		&\nabla\times\nabla\times\E-k^2\E=\boldsymbol{0} \quad && \text{in}~\Omega, \\
		&\bm\nu\times \E=\mathbf{0} \quad && \mbox{on}~\Gamma^\pm, \\
		&\bm\nu\times\nabla\times\E-\mathrm{i}\eta\E_T=\boldsymbol{f} \quad && \text{on}~\partial D.
		\end{aligned}
		\right.
\end{equation}
Here $\bm{\nu} $ represents the unit outward normal vector to $ \Gamma^\pm $ and $ \partial D $, and $\E_T=\bm\nu\times(\bm E\times\bm\nu)$. The impedance coefficient $ \eta $ is set to be a positive constant.

To ensure the uniqueness of the problem (\ref{model:E}), an appropriate radiation condition at infinity should be imposed. Denote by $(r,\theta,z)$ the cylindrical coordinate with $x_1=r\cos\theta, x_2=r\sin\theta$ and $r=\sqrt{x_1^2+x_2^2}$. Letting $E_r=\E\cdot\boldsymbol{e}_r$, $E_\theta=\E\cdot\boldsymbol{e}_\theta$, $E_z=\E\cdot\boldsymbol{e}_z$ with $\boldsymbol{e}_r=(\cos\theta, \sin\theta,0)^\top$, $\boldsymbol{e}_\theta = (-\sin\theta,\cos\theta,0)^\top$, $\boldsymbol{e}_z=(0,0,1)^\top$, the boundary condition $\bm\nu\times\bm E=0$ on $\Gamma^\pm$ implies that $E_r=E_\theta=0$ on $\Gamma^\pm$. Moreover, it can be deduced from the divergence free condition $\nabla\cdot\bm E=0$ that $\partial_z E_z=0$ on $\Gamma^\pm$. Then for any $\bm x\in \R^3_Z$ with $r>r_D:= \max_{\bm x\in \partial D} \sqrt{x_1^2+x_2^2}$, the field $\bm E$ can be expressed as
\begin{align}
\label{fieldexp}
\begin{bmatrix}
E_r\\ E_\theta\\ E_z
\end{bmatrix}
	=\sum_{m=0}^{\infty}\sum_{n=-\infty}^{\infty} \begin{bmatrix}
	    E_{r,nm}(r)\sin(m\pi z/Z)\\
        E_{\theta,nm}(r)\sin(m\pi z/Z)\\
        E_{z,nm}(r)\cos(m\pi z/Z)
	\end{bmatrix}e^{\mathrm{i}n\theta}.
\end{align}
Then the radiation condition for the considered parallel-plate waveguide problem is to impose the mode expansion coefficients $ \E_{nm}=(E_{r,nm},E_{\theta,nm},E_{z,nm})^\top $ to satisfy
\begin{equation}\label{radiation1}
	\partial_r\E_{nm}-\mathrm{i}k_m\E_{nm}=o(r^{-1/2})\quad\mbox{as}\quad r\rightarrow\infty,
\end{equation}
where $k_m=(k^2-m^2\pi^2/Z^2)^{1/2}$ if $k\ge m\pi/Z$ and $k_m=\mathrm{i}(m^2\pi^2/Z^2-k^2)^{1/2}$ if $k< m\pi/Z$. For simplicity, we assume that $ k \neq m\pi/Z $ holds for all $ m\geq0 $.
\begin{remark}
In particular, from the solution expansion discussed in Section~\ref{sec 3} and the asymptotic behavior of the modified Bessel function, it is known that for $k<m\pi/Z$, each mode coefficient $\E_{nm}$ is exponential decay, i.e.,
\begin{equation}\label{radiation2}
\E_{nm}=O(e^{-|k_m|r}) \quad\mbox{as}\quad r\rightarrow\infty\quad \text{for}~k<m\pi/Z.
\end{equation}
\end{remark}

\begin{theorem}
The electromagnetic scattering problem (\ref{model:E})-(\ref{radiation2}) has at most one solution.
\end{theorem}
\begin{proof}
Let $ B_R =\{\bm{x}\in\R_Z^3:r<R\} $ be a cylinder with radius $ R>r_D $ and denote $\Gamma_R=\{\bm{x}\in\R_Z^3:r=R\}$. It is sufficient to show that the problem (\ref{model:E})-(\ref{radiation2}) only admits trivial solution if $ \bm{f}=\mathbf{0} $. Multiplying (\ref{model:E}) by $ \overline{\bm{E}} $, integrating on $ B_R\backslash\overline{D} $, taking integration by parts and substituting boundary conditions yields
    \begin{equation}\label{integration by parts}
		\int_{B_R\backslash\overline{D}}|\nabla\times\bm{E}|^2-k^2|\bm{E}|^2\mathrm{d}\bm{x}
		-\mathrm{i}\eta\int_{\partial D}|\bm{E}_T|^2\mathrm{d}s
		+\int_{\Gamma_R}\bm{\nu}\times\nabla\times\bm{E}\cdot\overline{\bm{E}}_T\mathrm{d}s
		=0.
	\end{equation}
	Taking the imaginary part of (\ref{integration by parts}), we have
	\begin{equation}\label{ImBT}
		-\eta\int_{\partial D}|\bm{E}_T|^2\mathrm{d}s+\mathrm{Im}\int_{\Gamma_R}\bm{\nu}\times\nabla\times\bm{E}\cdot\overline{\bm{E}}_T=0.
	\end{equation}
    It follows from the series expansions (\ref{fieldexp}) that
	\begin{align*}
		&\quad \int_{\Gamma_R}(\bm{\nu}\times\nabla\times\bm{E})\cdot\overline{\bm{E}}_T\mathrm{d}s \\
		&=\int_{\Gamma_R}\left(-\frac{1}{r}\left[\partial_r(rE_\theta)-\partial_{\theta}E_r\right]\bm{e}_\theta+\left[\partial_zE_r-\partial_rE_z\right]\bm{e}_z\right)\cdot\left(\overline{E_\theta}\bm{e}_\theta+\overline{E_z}\bm{e}_z\right)\mathrm{d}s \\
		&=\int_{0}^{Z}\int_{0}^{2\pi}\left(-|E_\theta|^2-R\partial_rE_\theta\overline{E_\theta}+\partial_\theta E_r\overline{E_\theta}+R\partial_z E_r\overline{E_z}-R\partial_r E_z\overline{E_z}\right)\mathrm{d}\theta\mathrm{d}z \\
		&=\pi Z\sum_{m=0}^{\infty}\sum_{n=-\infty}^{\infty}\big(-|E_{\theta,nm}|^2-RE_{\theta,nm}^\prime\overline{E_{\theta,nm}}+\mathrm{i}nE_{r,nm}\overline{E_{\theta,nm}} \\
        &\quad+R\frac{m\pi}{Z}E_{r,nm}\overline{E_{z,nm}}-RE_{z,nm}^\prime\overline{E_{z,nm}}\big).
	\end{align*}
    The divergence free condition gives
	\begin{equation*}
		\frac{1}{r}\frac{\partial(rE_r)}{\partial r}+\frac{1}{r}\frac{\partial E_\theta}{\partial\theta}+\frac{\partial E_z}{\partial z}=0,
	\end{equation*}
	i.e.,
	\begin{align*}
		\frac{1}{r}\left(E_{r,nm}+r\partial_rE_{r,nm}+\mathrm{i}nE_{\theta,nm}\right)-\frac{m\pi}{Z}E_{z,nm}&=0.
	\end{align*}
Then we get
\begin{align*}
&\int_{\Gamma_R}\bm{\nu}\times\nabla\times\bm{E}\cdot\overline{\bm{E}}_T\mathrm{d}s \\
=&~\pi Z\sum_{m=0}^{\infty}\sum_{n=-\infty}^{\infty}\bigg(\left[-E_{\theta,nm}-R\partial_rE_{\theta,nm}+\mathrm{i}nE_{r,nm}\right]\overline{E_{\theta,nm}} \\		&+E_{r,nm}\left[\overline{E_{r,nm}}+R\overline{\partial_rE_{r,nm}}-\mathrm{i}n\overline{E_{\theta,nm}}\right]-R\partial_rE_{z,nm}\overline{E_{z,nm}}\bigg) \\
=&~\pi Z\sum_{m=0}^{\infty}\sum_{n=-\infty}^{\infty}\left(|E_{r,nm}|^2-|E_{\theta,nm}|^2\right) \\
&+\pi Z\sum_{m=0}^{\infty}\sum_{n=-\infty}^{\infty}R\left(E_{r,nm}\overline{\partial_rE_{r,nm}}-\partial_rE_{\theta,nm}\overline{E_{\theta,nm}}-\partial_rE_{z,nm}\overline{E_{z,nm}}\right),
\end{align*}
which implies that
	\begin{equation}\label{sums of ImBT}
		\mathrm{Im}\int_{\Gamma_R}\bm{\nu}\times\nabla\times\bm{E}\cdot\overline{\bm{E}}_T\mathrm{d}s
		=-\pi RZ\sum_{m=0}^{\infty}\sum_{n=-\infty}^{\infty}\sum_{\xi\in\{r,\theta,z\}}\mathrm{Im}[\partial_rE_{\xi,nm}\overline{E_{\xi,nm}}].
	\end{equation}
For $ k>m\pi/Z $ with $k_m>0$, we know
    \begin{align*}
        &\left(\partial_rE_{\xi,nm}-\mathrm{i}k_mE_{\xi,nm}\right)\overline{\left(\partial_rE_{\xi,nm}-\mathrm{i}k_mE_{\xi,nm}\right)} \\
        &=\left(|\partial_rE_{\xi,nm}|^2+k_m^2|E_{\xi,nm}|^2-2k_m\mathrm{Im}[\partial_rE_{\xi,nm}\overline{E_{\xi,nm}}]\right),
    \end{align*}
    which, together with the radiation condition (\ref{radiation1}), give
    \begin{equation}\label{sum1 of ImBT}
        \lim_{R\rightarrow\infty}2R\mathrm{Im}[\partial_rE_{\xi,nm}\overline{E_{\xi,nm}}]
        \geq0.
    \end{equation}
For $ k<m\pi/Z $ with $ k_m=\mathrm{i}|k_m| $, it easily follows from (\ref{radiation2}) that
    \begin{equation}\label{sum2 of ImBT}
        \lim_{R\to\infty}R\mathrm{Im}[\partial_rE_{\xi,nm}\overline{E_{\xi,nm}}]
        =0.
    \end{equation}
    Combining (\ref{ImBT})-(\ref{sum2 of ImBT}) with the impedance boundary condition in (\ref{model:E}), we obtain that
    \begin{equation*}
        \bm{\nu}\times\nabla\times\bm{E}
        =\bm{E}_T
        =\mathbf{0} \quad \text{on}~\partial D.
    \end{equation*}
Then an application of the unique continuation result \cite[Theorem 4.13]{MR2059447} completes the proof.
\end{proof}

\subsection{Sobolev spaces}
\label{sec:2.2}

The study of the wellposedness of the electromagnetic scattering problem (\ref{model:E})-(\ref{radiation2}) relies on an appropriate transparent boundary condition to truncate the infinite domain as well as the Sobolev spaces on the truncated domain whose notations are introduced as follows. 

Letting $\Gamma_R=\{\boldsymbol{x}\in\Omega:x_1^2+x_2^2=R^2\} $ and $ \Gamma_R^{\pm}=\{\boldsymbol{x}\in\Gamma^{\pm}:x_1^2+x_2^2<R^2\}$, we 
denote the truncated domain by $C_R=\{\boldsymbol{x}\in\mathbb{R}^3:x_1^2+x_2^2<R^2, 0<x_3<Z\} $ with $R>0$ such that $\overline{D}\subset C_R$. Denote $ \Omega_R=C_R\backslash\overline{D} $ and $ \Sigma_R=\Gamma_R^+\cup\Gamma_R^- $. Let $ \boldsymbol{H}(\mathrm{curl},\Omega_R)\subset\boldsymbol{L}^2(\Omega_R) $ be the subspace whose functions have square-integrable curls where $ \boldsymbol{L}^2(\Omega_R):=L^2(\Omega_R)^3 $. Throughout this paper, we denote vector-valued quantities by boldface notion. For any $v\in L^2(\Gamma_R)$ which admits a Fourier series expansion 
\begin{equation*}
	v=\sum_{m=0}^{\infty}\sum_{n=-\infty}^{\infty}[v_{c,nm}\cos(m\pi z/Z)+v_{s,nm}\sin(m\pi z/Z)]e^{\mathrm{i}n\theta},
\end{equation*}
an equivalent norm of $v$ on $L^2(\Gamma_R)$ can be obtained from the Parseval identity as
\begin{equation*}
	\|v\|_{L^2(\Gamma_R)}^2=\sum_{m=0}^{\infty}\sum_{n=-\infty}^{\infty}\left(|v_{c,nm}|^2+|v_{s,nm}|^2\right).
\end{equation*}
In addition, the norm of $v$ on $H^s(\Gamma_R)$ can be defined as
\begin{equation*}
	\|v\|_{H^s(\Gamma_R)}^2=\sum_{m=0}^{\infty}\sum_{n=-\infty}^{\infty}(1+n^2+m^2)^s\left(|v_{c,nm}|^2+|v_{s,nm}|^2\right).
\end{equation*}

Moreover, the following tangential functional spaces are defined:
\begin{align*}
	\boldsymbol{L}_t^2(\Gamma_R)
	&=\{\boldsymbol{v}\in\boldsymbol{L}^2(\Gamma_R):\boldsymbol{v}\cdot\boldsymbol{e}_r=0\}, \\
	\boldsymbol{H}_t^s(\Gamma_R)
	&=\{\boldsymbol{v}\in\boldsymbol{H}^s(\Gamma_R):\boldsymbol{v}\cdot\boldsymbol{e}_r=0\}, \\
	\boldsymbol{H}^{-1/2}(\mathrm{Div},\Gamma_R)
	&=\{\boldsymbol{v}\in\boldsymbol{H}_t^{-1/2}(\Gamma_R):\nabla_{\Gamma_R}\cdot\boldsymbol{v} \in H^{-1/2}(\Gamma_R)\}, \\
	\boldsymbol{H}^{-1/2}(\mathrm{Curl},\Gamma_R)
	&=\{\boldsymbol{v}\in\boldsymbol{H}_t^{-1/2}(\Gamma_R):\nabla_{\Gamma_R}\times\boldsymbol{v} \in H^{-1/2}(\Gamma_R)\}.
\end{align*}
Here the surface divergence and the surface scalar curl on $ \Gamma_R $ are denoted by
\begin{equation*}
	\nabla_{\Gamma_R}\cdot\boldsymbol{u}=\frac{1}{R}\partial_\theta u_\theta+\partial_zu_z, \quad \nabla_{\Gamma_R}\times\boldsymbol{u}=\frac{1}{R}\partial_\theta u_z-\partial_zu_\theta.
\end{equation*}
For any $\bm v\in \boldsymbol{H}^{-1/2}(\mathrm{Div},\Gamma_R)$ or $\bm v\in \boldsymbol{H}^{-1/2}(\mathrm{Curl},\Gamma_R)$ whose expansion is given by
\begin{align*}
\bm{v}
=&~\sum_{m=0}^{\infty}\sum_{n=-\infty}^{\infty}[v_{\theta,c,nm}\cos(m\pi z/Z)+v_{\theta,s,nm}\sin(m\pi z/Z)]e^{\mathrm{i}n\theta}\bm{e}_\theta \\
&+\sum_{m=0}^{\infty}\sum_{n=-\infty}^{\infty}[v_{z,c,nm}\cos(m\pi z/Z)+v_{z,s,nm}\sin(m\pi z/Z)]e^{\mathrm{i}n\theta}\bm{e}_z,
\end{align*}
its norm on $ \boldsymbol{H}^{-1/2}(\mathrm{Div},\Gamma_R) $ or $ \boldsymbol{H}^{-1/2}(\mathrm{Curl},\Gamma_R) $ can be characterized, respectively, as
\begin{align*}
	&\|\boldsymbol{v}\|_{\boldsymbol{H}^{-1/2}(\mathrm{Div},\Gamma_R)}^2 \\
	&=\sum_{m=0}^{\infty}\sum_{n=-\infty}^{\infty}
    (1+n^2+m^2)^{-1/2}\bigg(|v_{\theta,s,nm}|^2+|v_{\theta,c,nm}|^2+|v_{z,s,nm}|^2+|v_{z,c,nm}|^2 \\
    &\quad+\left|\frac{\mathrm{i}n}{R}v_{\theta,s,nm}-\frac{m\pi}{Z}v_{z,c,nm}\right|^2+\left|\frac{\mathrm{i}n}{R}v_{\theta,c,nm}+\frac{m\pi}{Z}v_{z,s,nm}\right|^2\bigg), \\
	&\|\boldsymbol{v}\|_{\boldsymbol{H}^{-1/2}(\mathrm{Curl},\Gamma_R)}^2 \\
	&=\sum_{m=0}^{\infty}\sum_{n=-\infty}^{\infty}
    (1+n^2+m^2)^{-1/2}\bigg(|v_{\theta,s,nm}|^2+|v_{\theta,c,nm}|^2+|v_{z,s,nm}|^2+|v_{z,c,nm}|^2 \\
    &\quad+\left|\frac{\mathrm{i}n}{R}v_{z,s,nm}+\frac{m\pi}{Z}v_{\theta,c,nm}\right|^2+\left|\frac{\mathrm{i}n}{R}v_{z,c,nm}-\frac{m\pi}{Z}v_{\theta,s,nm}\right|^2\bigg).
\end{align*}

The following lemma states that $ \boldsymbol{H}^{-1/2}(\mathrm{Div},\Gamma_R) $ and $ \boldsymbol{H}^{-1/2}(\mathrm{Curl},\Gamma_R) $ are dual spaces with respect to the scalar product in $ \boldsymbol{L}_t^2(\Gamma_R) $. In the following, denote by $ \langle\cdot,\cdot\rangle_{\Gamma_R} $ the duality pairing between $ \boldsymbol{H}^{-1/2}(\mathrm{Div},\Gamma_R)$ and $ \boldsymbol{H}^{-1/2}(\mathrm{Curl},\Gamma_R)$ with
\begin{align*}
    \langle\bm{u},\bm{v}\rangle_{\Gamma_R}
    =\pi RZ\sum_{m=0}^{\infty}\sum_{n=-\infty}^{\infty}
    &\big(u_{\theta,s,nm}\overline{v}_{\theta,s,nm}+u_{\theta,c,nm}\overline{v}_{\theta,c,nm} \\
    &+u_{z,s,nm}\overline{v}_{z,s,nm}+u_{z,c,nm}\overline{v}_{z,c,nm}\big).
\end{align*}
\begin{lemma}
\label{lem-dual}
    The spaces $ \boldsymbol{H}^{-1/2}(\mathrm{Div},\Gamma_R) $ and $ \boldsymbol{H}^{-1/2}(\mathrm{Curl},\Gamma_R) $ are mutually adjoint with respect to the scalar product in $ \boldsymbol{L}_t^2(\Gamma_R) $.
\end{lemma}
\begin{proof}
The proof is given in Appendix~\ref{appendixA}.
\end{proof}

We know from \cite{MR1944792} that the following surjective mapping properties
\begin{align*}
	\gamma&:H^1(\Omega_R) \to H^{1/2}(\partial\Omega_R), \quad &&\gamma u=u~\text{on}~\partial\Omega_R, \\
	\gamma_t&:\boldsymbol{H}(\mathrm{curl},\Omega_R)\to\boldsymbol{H}^{-1/2}(\mathrm{Div},\partial\Omega_R), \quad &&\gamma_t\boldsymbol{u}=\boldsymbol{n}\times\boldsymbol{u}~\text{on}~\partial\Omega_R, \\
	\gamma_T&:\boldsymbol{H}(\mathrm{curl},\Omega_R)\to\boldsymbol{H}^{-1/2}(\mathrm{Curl},\partial\Omega_R), \quad &&\gamma_T\boldsymbol{u}=\boldsymbol{n}\times(\boldsymbol{u}\times\boldsymbol{n})~\text{on}~\partial\Omega_R.
\end{align*}
In addition, we denote
\begin{align*}
	H_{\Sigma_R}^1(\Omega_R)
	&=\left\{u \in H^1(\Omega_R):\gamma u=0~\text{on}~\Sigma_R\right\}, \\
	\boldsymbol{H}_{\Sigma_R}(\mathrm{curl},\Omega_R)
	&=\left\{\boldsymbol{u}\in\boldsymbol{H}(\mathrm{curl},\Omega_R):\gamma_t\boldsymbol{u}=\mathbf{0}~\text{on}~\Sigma_R\right\},
\end{align*}
and
\begin{equation*}
	\boldsymbol{X}=\left\{\boldsymbol{u}\in\boldsymbol{H}(\mathrm{curl},\Omega_R):\gamma_t\boldsymbol{u}=\mathbf{0}~\text{on}~\Sigma_R,~\gamma_T\boldsymbol{u}\in\boldsymbol{L}_t^2(\Gamma_D)~\text{on}~\Gamma_D\right\}.
\end{equation*}
equipped with the norm
\begin{equation*}
	\|\boldsymbol{u}\|_{\boldsymbol{X}}^2
	=\|\boldsymbol{u}\|_{\boldsymbol{H}(\mathrm{curl},\Omega_R)}^2+\|\gamma_T\boldsymbol{u}\|_{\boldsymbol{L}_t^2(\Gamma_D)}^2.
\end{equation*}

\section{Wellposedness analysis of the direct scattering problem}\label{sec 3}

Now we investigate the wellposedness of the electromagnetic scattering problem (\ref{model:E})-(\ref{radiation2}) by introducing an appropriate TBC to truncate the infinite domain. 

\subsection{TBC}
\label{sec:3.1}

The required TBC
\begin{align}
\label{TBC}
\mathscr{T}[\bm{\nu}\times\bm{E}]=\bm{\nu}\times\nabla\times\bm{E} \quad \text{on}~\Gamma_R,
\end{align}
is introduced based on an electro-to-magnetic Calder\'on operator $\mathscr{T}$ defined as follows.
\begin{definition}
For any $ \bm{g}\in\bm{H}^{-1/2}(\mathrm{Div},\Gamma_R) $, the Calder\'on operator $\mathscr{T}$ acting on $ \bm{g} $ is defined as
	\begin{equation*}
		\mathscr{T}[\bm{g}]=(\bm{\nu}\times\nabla\times\bm{E})|_{\Gamma_R},
	\end{equation*}
where $ \bm{E}\in\bm{H}_{loc}(\R_Z^2\backslash\overline{C_R}) $ is the unique radiating solution to the boundary value problem
	\begin{align*}
		\nabla\times\nabla\times\bm{E}-k^2\bm{E}&=\mathbf{0} \quad \text{in}~\R_Z^2\backslash\overline{C_R}, \\
		\bm{\nu}\times\bm{E}&=\mathbf{0} \quad \text{on}~\Gamma^\pm\backslash\Gamma_R^\pm, \\
		\bm{\nu}\times\bm{E}&=\bm{g} \quad \text{on}~\Gamma_R,
	\end{align*}
	together with the radiation condition (\ref{radiation1})-(\ref{radiation2}).
\end{definition}

Considering the vector Helmholtz equation
\begin{equation*}
	\Delta\boldsymbol{E}+k^2\boldsymbol{E}=\mathbf{0} \quad r \geq R,
\end{equation*}
and letting $ E_{\pm,nm}=E_{r,nm}\pm\mathrm{i}E_{\theta,nm} $, the field expansion 
(\ref{fieldexp}) gives
\begin{align*}
r^2\frac{\mathrm{d}^2E_{\pm,nm}}{\mathrm{d}r^2}+r\frac{\mathrm{d}E_{\pm,nm}}{\mathrm{d}r}+\left[k_m^2r^2-(n\pm1)^2\right]E_{\pm,nm}&=0, \\
	r^2\frac{\mathrm{d}^2E_{z,nm}}{\mathrm{d}r^2}+r\frac{\mathrm{d}E_{z,nm}}{\mathrm{d}r}+\left[k_m^2r^2-n^2\right]E_{z,nm}&=0,
\end{align*}
Then utilizing the radiation condition (\ref{radiation1})-(\ref{radiation2}) results into the following series expansions:
\begin{align}
	E_r(r,\theta,z)
	&=\frac{1}{2}\sum_{m=0}^{\infty}\sum_{n=-\infty}^{\infty}\left[A_{nm}^+H_{n+1}^{(1)}(k_mr)+A_{nm}^-H_{n-1}^{(1)}(k_mr)\right]e^{\mathrm{i}n\theta}\sin(m\pi z/Z), \label{series expansion of Er} \\
	E_\theta(r,\theta,z)
	&=\frac{1}{2\mathrm{i}}\sum_{m=0}^{\infty}\sum_{n=-\infty}^{\infty}\left[A_{nm}^+H_{n+1}^{(1)}(k_mr)-A_{nm}^-H_{n-1}^{(1)}(k_mr)\right]e^{\mathrm{i}n\theta}\sin(m\pi z/Z), \label{series expansion of Etheta} \\
	E_z(r,\theta,z)
	&=\sum_{m=0}^{\infty}\sum_{n=-\infty}^{\infty}B_{nm}H_n^{(1)}(k_mr)e^{\mathrm{i}n\theta}\cos(m\pi z/Z). \label{series expansion of Ez}
\end{align}
Here $ H_n^{(1)} $ denotes the Hankel function of the first kind and of order $ n $.

Next, we derive the expression $ \bm\nu\times(\nabla\times\boldsymbol{E}) $ on $ \Gamma_R $. Noting that $\bm\nu=\bm e_r$ on $\Gamma_R$, it holds that
\begin{align*}
&(\bm\nu\times\nabla\times\boldsymbol{E})\cdot\boldsymbol{e}_\theta\big|_{\Gamma_R} \\
	=&\sum_{m=0}^{\infty}\sum_{n=-\infty}^{\infty}\left(\frac{\mathrm{i}k_m}{2}H_n^{(1)}(k_mR)A_{nm}^++\frac{\mathrm{i}k_m}{2}H_n^{(1)}(k_mR)A_{nm}^-\right)e^{\mathrm{i}n\theta}\sin(m\pi z/Z), \\
	&(\bm{\nu}\times\nabla\times\boldsymbol{E})\cdot\boldsymbol{e}_z\big|_{\Gamma_R} \\
	=&\sum_{m=0}^{\infty}\sum_{n=-\infty}^{\infty}\left(\frac{m\pi}{2Z}H_{n+1}^{(1)}(k_mR)A_{nm}^++\frac{m\pi}{2Z}H_{n-1}^{(1)}(k_mR)A_{nm}^-\right)e^{\mathrm{i}n\theta}\cos(m\pi z/Z) \\
	&-\sum_{m=0}^{\infty}\sum_{n=-\infty}^{\infty}k_mH_n^{(1)\prime}(k_mR)B_{nm}e^{\mathrm{i}n\theta}\cos(m\pi z/Z).
\end{align*}
Denote $S_{nm}=e^{\mathrm{i}n\theta}\sin(m\pi z/Z)$, $\boldsymbol{S}_{\alpha,nm}=S_{nm}\boldsymbol{e}_\alpha$, $C_{nm}=e^{\mathrm{i}n\theta}\cos(m\pi z/Z)$ and $\boldsymbol{C}_{\alpha,nm}=C_{nm}\boldsymbol{e}_\alpha$ with $ \alpha=r,\theta,z $. Define the $ L^2 $ inner product on $ \Gamma_R $ by
\begin{equation*}
    \left(\boldsymbol{u},\boldsymbol{v}\right)_{\Gamma_R}
    :=\frac{1}{\pi RZ}\int_{\Gamma_R}\boldsymbol{u}\cdot\overline{\boldsymbol{v}} \mathrm{d}s.
\end{equation*}
Hence,
\begin{align*}
	&\begin{bmatrix}
		\left((\boldsymbol{\nu}\times\nabla\times \boldsymbol{E})|_{\Gamma_R}, \boldsymbol{S}_{\theta,nm}\right)_{\Gamma_R}\\
		\left((\boldsymbol{\nu}\times\nabla\times\boldsymbol{E})|_{\Gamma_R},\boldsymbol{C}_{z.nm}\right)_{\Gamma_R}
	\end{bmatrix} \\
&=
\begin{bmatrix}
	\frac{\mathrm{i}k_m}{2}H_n^{(1)}(k_mR) & \frac{\mathrm{i}k_m}{2}H_n^{(1)}(k_mR) & 0 \\
	\frac{m\pi}{2Z}H_{n+1}^{(1)}(k_mR) & \frac{m\pi}{2Z}H_{n-1}^{(1)}(k_mR) & -k_mH_n^{(1)\prime}(k_mR)
\end{bmatrix}
\begin{bmatrix}
	A_{nm}^+ \\
	A_{nm}^- \\
	B_{nm}
\end{bmatrix}
:=P_{nm}
\begin{bmatrix}
	A_{nm}^+ \\
	A_{nm}^- \\
	B_{nm}
\end{bmatrix},
\end{align*}
and
\begin{align*}
    &\begin{bmatrix}
		\left((\boldsymbol{\nu}\times\boldsymbol{E})|_{\Gamma_R},\boldsymbol{S}_{z,nm}\right)_{\Gamma_R} \\
		\left((\boldsymbol{\nu}\times\boldsymbol{E})|_{\Gamma_R},\boldsymbol{C}_{\theta,nm}\right)_{\Gamma_R} \\
		\left(\nabla\cdot\boldsymbol{E}|_{\Gamma_R},S_{nm}\right)_{\Gamma_R}
	\end{bmatrix} \\
&=
\begin{bmatrix}
	\frac{1}{2\mathrm{i}}H_{n+1}^{(1)}(k_mR) & -\frac{1}{2\mathrm{i}}H_{n-1}^{(1)}(k_mR) & 0 \\
	0 & 0 & -H_n^{(1)}(k_mR) \\
	\frac{k_m}{2}H_n^{(1)}(k_mR) & -\frac{k_m}{2}H_n^{(1)}(k_mR) & -\frac{m\pi}{Z}H_n^{(1)}(k_mR)
\end{bmatrix}
\begin{bmatrix}
	A_{nm}^+ \\
	A_{nm}^- \\
	B_{nm}
\end{bmatrix}
:=Q_{nm}
\begin{bmatrix}
	A_{nm}^+ \\
	A_{nm}^- \\
	B_{nm}
\end{bmatrix}.
\end{align*}

\begin{lemma}\label{le 3.1}
	The matrix $ Q_{nm} $ is invertible for all $ m\in\mathbb{N} $ and $ n\in\mathbb{Z} $. Its inverse is given by
	\begin{equation*}
		Q_{nm}^{-1}=
		\begin{bmatrix}
			-\frac{\mathrm{i}}{H_n^{(1)\prime}(k_mR)} & -\frac{m\pi}{Z}\frac{H_{n-1}^{(1)}(k_mR)}{k_mH_n^{(1)}(k_mR)H_n^{(1)\prime}(k_mR)} & \frac{H_{n-1}^{(1)}(k_mR)}{k_mH_n^{(1)}(k_mR)H_n^{(1)\prime}(k_mR)} \\
			-\frac{\mathrm{i}}{H_n^{(1)\prime}(k_mR)} & -\frac{m\pi}{Z}\frac{H_{n+1}^{(1)}(k_mR)}{k_mH_n^{(1)}(k_mR)H_n^{(1)\prime}(k_mR)} & \frac{H_{n+1}^{(1)}(k_mR)}{k_mH_n^{(1)}(k_mR)H_n^{(1)\prime}(k_mR)} \\
			0 & -\frac{1}{H_n^{(1)}(k_mR)} & 0
		\end{bmatrix}.
	\end{equation*}
\end{lemma}
\begin{proof}
	It's sufficient to prove that $ \det(Q_{nm})\neq0 $. A straightforward calculations gives
	\begin{align*}
		\det(Q_{nm})
		&=\frac{k_m}{4\mathrm{i}}(H_n^{(1)}(k_mR))^2(H_{n-1}^{(1)}(k_m)-H_{n+1}^{(1)}(k_mR)) \\
		&=\frac{k_m}{2\mathrm{i}}(H_n^{(1)}(k_mR))^2H_n^{(1)\prime}(k_mR).
	\end{align*}
For $ m<kZ/\pi $, we know from \cite{MR2988887} that
\begin{equation*}
	H_n^{(1)\prime}(k_mR)\neq0, \quad H_n^{(1)}(k_mR)\neq0.
\end{equation*}
For $ m>kZ/\pi $, letting $ K_n $ denote the modified Bessel function, it holds that
\begin{equation*}
	H_n^{(1)}(\mathrm{i}t)=\frac{2}{\pi}e^{-\mathrm{i}(n+1)\pi/2}K_n(t) \quad t>0.
\end{equation*}
Then it follows from \cite[Appendix C]{MR3376143} that
\begin{align*}
	\det(Q_{nm})
	=\mathrm{i}\frac{2|k_m|}{\pi^3}e^{-\mathrm{i}3(n+1)\pi/2}(K_n(|k_m|R))^2(K_{n-1}(|k_m|R)+K_{n+1}(|k_m|R))
	\neq0.
\end{align*}
Hence, $ \det(Q_{nm})\neq0 $ holds for all $ m\in\mathbb{N} $ and $ n\in\mathbb{Z} $.
\end{proof}

Therefore, we arrive at
\begin{align*}
	\begin{bmatrix}
    \left((\boldsymbol{\nu}\times\nabla\times\boldsymbol{E})|_{\Gamma_R},\boldsymbol{S}_{\theta,nm}\right)_{\Gamma_R} \\
    \left((\boldsymbol{\nu}\times\nabla\times\boldsymbol{E})|_{\Gamma_R}\boldsymbol{C}_{z,nm}\right)_{\Gamma_R}
	\end{bmatrix}
=&~P_{nm}Q_{nm}^{-1}
\begin{bmatrix}
	\left((\boldsymbol{\nu}\times\boldsymbol{E})|_{\Gamma_R},\boldsymbol{S}_{z,nm}\right)_{\Gamma_R} \\
	\left((\boldsymbol{\nu}\times\boldsymbol{E})|_{\Gamma_R},\boldsymbol{C}_{\theta,nm}\right)_{\Gamma_R} \\
	0
\end{bmatrix} \\
:=&~W_{nm}
\begin{bmatrix}
\left((\boldsymbol{\nu}\times\boldsymbol{E})|_{\Gamma_R},\boldsymbol{S}_{z,nm}\right)_{\Gamma_R} \\
	\left((\boldsymbol{\nu}\times\boldsymbol{E})|_{\Gamma_R},\boldsymbol{C}_{\theta,nm}\right)_{\Gamma_R}
\end{bmatrix}
\end{align*}
where
\begin{align*}
&W_{nm}^{11}
=\frac{k_mH_n^{(1)}(k_mR)}{H_n^{(1)\prime}(k_mR)}, \quad
W_{nm}^{12}=W_{nm}^{21}
=-\frac{\mathrm{i}nm\pi}{RZ}\frac{H_n^{(1)}(k_mR)}{k_mH_n^{(1)\prime}(k_mR)}, \\
&W_{nm}^{22}
=-\frac{n^2m^2\pi^2}{k_m^2R^2Z^2}\frac{H_n^{(1)}(k_mR)}{k_mH_n^{(1)\prime}(k_mR)}+k^2\frac{H_n^{(1)\prime}(k_mR)}{k_mH_n^{(1)}(k_mR)}.
\end{align*}
Given a field $\bm\nu\times\boldsymbol{U}\in \boldsymbol{H}^{-1/2}(\mathrm{Div},\Gamma_R)$ on $ \Gamma_R $, the Calder\'{o}n operator $ \mathscr{T}$ hence can be explicitly expressed as
\begin{equation}\label{definition of Calderon operator T}
	\begin{aligned}
		\mathscr{T}[\boldsymbol{\nu}\times\boldsymbol{U}]
        =&\sum_{n=-\infty}^{\infty}\sum_{m=0}^{\infty}[\boldsymbol{S}_{\theta,nm},\boldsymbol{C}_{z,nm}]W_{nm}
        \begin{bmatrix}
            \left((\boldsymbol{\nu}\times\boldsymbol{U})|_{\Gamma_R},\boldsymbol{S}_{z,nm}\right)_{\Gamma_R} \\
            \left((\boldsymbol{\nu}\times\boldsymbol{U})|_{\Gamma_R},\boldsymbol{C}_{\theta,nm}\right)_{\Gamma_R}
        \end{bmatrix}.
	\end{aligned}
\end{equation}

Utilizing the TBC (\ref{TBC}), the scattering problem of electric field $ \boldsymbol{E} $ can be reformulated on the truncated domain as
\begin{equation}\label{truncated scattering problem}
	\left\{
	\begin{aligned}
		& \nabla\times\nabla\times\boldsymbol{E}-k^2\boldsymbol{E}=\mathbf{0} \quad && \text{in}~\Omega_R, \\
		&\boldsymbol{\nu}\times\nabla\times\boldsymbol{E}-\mathrm{i}\eta\boldsymbol{E}_T=\boldsymbol{f} \quad && \text{on}~\partial D, \\
		&\boldsymbol{\nu}\times\nabla\times\boldsymbol{E}=\mathscr{T}[\boldsymbol{\nu}\times\boldsymbol{E}] \quad && \text{on}~\Gamma_R, \\
		& \boldsymbol{\nu}\times\boldsymbol{E}=\mathbf{0} \quad && \text{on}~\Sigma_R.
	\end{aligned}
	\right.
\end{equation}
Multiplying both sides of (\ref{truncated scattering problem}) with $ \boldsymbol{V}\in\boldsymbol{X} $, integrating the result on $ \Omega_R $, and taking integration by parts, the corresponding variational problem of (\ref{truncated scattering problem}) reads: find $ \boldsymbol{E}\in\boldsymbol{X} $, such that
\begin{equation}\label{variational problem}
	a(\boldsymbol{E},\boldsymbol{V})
	=b(\bm V),
    \quad \forall\boldsymbol{V}\in\boldsymbol{X},
\end{equation}
where the sesquilinear form $ a(\cdot,\cdot):\boldsymbol{X}\times\boldsymbol{X}\to\mathbb{C} $ is defined as
\begin{equation*}
	a(\boldsymbol{U},\boldsymbol{V})
	=\int_{\Omega_R}\left(\nabla\times\boldsymbol{U}\cdot\nabla\times\overline{\boldsymbol{V}}-k^2\boldsymbol{U}\cdot\overline{\boldsymbol{V}}\right)\mathrm{d}\boldsymbol{x}
	-\mathrm{i}\langle\eta\boldsymbol{U}_T\cdot\overline{\boldsymbol{V}}_T\rangle_{\partial D}
	+\langle\mathscr{T}[\boldsymbol{\nu}\times\boldsymbol{U}]\cdot\overline{\boldsymbol{V}}_T\rangle_{\Gamma_R},
\end{equation*}
and the the linear functional $b(\cdot):\boldsymbol{X}\to\mathbb{C}$ is given by
\begin{equation*}
b(\bm V)=
    \int_{\partial D}\boldsymbol{f}\cdot\overline{\boldsymbol{V}}_T\mathrm{d}s.
\end{equation*}

\subsection{Well-posedness analysis}
In this subsection, the well-posedness of variational problem (\ref{variational problem}) and the inf-sup condition for the sesquilinear form $ a(\cdot,\cdot) $ will be established. First we prove the following properties of the Calder\'{o}n operator $ \mathscr{T} $.

\begin{lemma}\label{continuity of T}
	There exists a constant $ C>0 $ depending only on $ k,R,Z $ such that
	\begin{equation*}
		\|\mathscr{T}[\boldsymbol{\nu}\times\boldsymbol{U}]\|_{\boldsymbol{H}^{-1/2}(\mathrm{Div},\Gamma_R)}
		\leq C\|\boldsymbol{\nu}\times\boldsymbol{U}\|_{\boldsymbol{H}^{-1/2}(\mathrm{Div},\Gamma_R)} \quad \forall \boldsymbol{\nu}\times\boldsymbol{U}\in\boldsymbol{H}^{-1/2}(\mathrm{Div},\Gamma_R).
	\end{equation*}
\end{lemma}
\begin{proof}
	A simple calculation yields
	\begin{align*}
		\langle\mathscr{T}[\boldsymbol{\nu}\times\boldsymbol{U}],\boldsymbol{V}_T\rangle_{\Gamma_R}
	=&~\pi RZ\sum_{m=0}^{\infty}\sum_{n=-\infty}^{\infty}\frac{H_n^{(1)}(k_mR)}{k_mH_n^{(1)\prime}(k_mR)}\Bigg[k^2U_{\theta,s,nm}\overline{V}_{\theta,s,nm} \\
    &+k^2\left(\frac{n^2}{k_m^2R^2}-\left(\frac{H_n^{(1)\prime}(k_mR)}{H_n^{(1)}(k_mR)}\right)^2\right)U_{z,c,nm}\overline{V}_{z,c,nm} \\
	&+\left(\frac{\mathrm{i}n}{R}U_{z,c,nm}-\frac{m\pi}{Z}U_{\theta,s,nm}\right)\left(\frac{m\pi}{Z}\overline{V}_{\theta,s,nm}+\frac{\mathrm{i}n}{R}\overline{V}_{z,c,nm}\right)\Bigg].
	\end{align*}
    It is sufficient to show that there exists a positive constant $ C $ such that
    \begin{equation}\label{inqs}
    \left|\frac{H_n^{(1)}(k_mR)}{k_mH_n^{(1)\prime}(k_mR)}\right|\leq C(1+n^2+m^2)^{-1/2}, \quad \left|\frac{n^2}{k_m^2R^2}-\left(\frac{H_n^{(1)\prime}(k_mR)}{H_n^{(1)}(k_mR)}\right)^2\right|\leq C,
    \end{equation}
    uniformly for all $ m\in\mathbb{N} $ and $ n\in\mathbb{Z} $.

For $ m<kZ/\pi$ and $ |n|>N $ with $N$ being sufficiently large, we know from the asymptotic behavior of Hankel functions (see, e.g., \cite[Chapter 10]{MR2723248}) that
\begin{align*}
	\left|\frac{H_n^{(1)}(k_mR)}{k_mH_n^{(1)\prime}(k_mR)}\right|
	=\left|-\frac{R}{|n|}+O(n^{-3})\right|
	\leq C|n|^{-1}.
\end{align*}
and
\begin{align*}
	\frac{n^2}{k_m^2R^2}-\frac{H_n^{(1)\prime}(k_mR)^2}{H_n^{(1)}(k_mR)^2}
	=\frac{n^2}{k_m^2R^2}-\left(-\frac{|n|}{k_mR}+\frac{k_mR}{2|n|}+O(n^{-2})\right)^2
	=O(1).
\end{align*}
which shows that (\ref{inqs}) holds for $ m<kZ/\pi$.

For $ m>kZ/\pi $, from the recurrence relations (cf. \cite[Appendix]{MR2402567})
\begin{equation*}
	\frac{H_n^{(1)\prime}(t)}{H_n^{(1)}(t)}=\frac{H_{|n|-1}^{(1)}(t)}{H_{|n|}^{(1)}(t)}-\frac{|n|}{t}, \quad \frac{H_{n-1}^{(1)}(\mathrm{i}t)}{H_n^{(1)}(\mathrm{i}t)}=\mathrm{i}\frac{K_{n-1}(t)}{K_n(t)},
\end{equation*}
it follows from \cite[(10.25.3), (10.41.2)]{MR2723248} for $ |n|\in\mathbb{N}_+ $ that
\begin{align*}
	&\left|\frac{n^2}{k_m^2R^2}-\left(\frac{H_n^{(1)\prime}(k_mR)}{H_n^{(1)}(k_mR)}\right)^2\right| \\
	=&~\left|\frac{n^2}{-|k_m|^2R^2}-\left(-\frac{K_{|n|-1}(|k_m|R)^2}{K_{|n|}(|k_m|R)^2}-2\frac{K_{|n|-1}(|k_m|R)}{K_{|n|}(|k_m|R)}\frac{|n|}{|k_m|R}+\frac{n^2}{k_m^2R^2}\right)\right| \\
	=&~\frac{K_{|n|-1}(|k_m|R)}{K_{|n|}(|k_m|R)}\frac{K_{|n|-1}(|k_m|R)+2\frac{|n|}{|k_m|R}K_{|n|}(|k_m|R)}{K_{|n|}(|k_m|R)} \\
	=&~\frac{K_{|n|-1}(|k_m|R)}{K_{|n|}(|k_m|R)}\frac{K_{|n|+1}(|k_m|R)}{K_{|n|}(|k_m|R)}
	\leq C.
\end{align*}
Using \cite[(A10)]{MR2402567}, it gives for $ n=0 $ that
\begin{align*}
	\left|\frac{H_0^{(1)\prime}(k_mR)^2}{H_0^{(1)}(k_mR)^2}\right|
	\leq\left|\frac{H_{-1}^{(1)}(k_mR)}{H_0^{(1)}(k_mR)}\right|^2
	\leq C.
\end{align*}
Moreover, for $ |n|\in\mathbb{N}$,
\begin{align*}
	\left|\frac{H_n^{(1)}(k_mR)}{k_mH_n^{(1)\prime}(k_mR)}\right|
	=\frac{R}{|k_m|RK_{|n|-1}(|k_m|R)/K_{|n|}(|k_m|R)+|n|}
	\leq C(|n|+m)^{-1},
\end{align*}
which completes the proof of (\ref{inqs}) holds for $ m>kZ/\pi$.

Therefore, applying the Cauchy-Schwarz inequality yields
\begin{align*}
	&|\langle\mathscr{T}[\boldsymbol{\nu}\times\boldsymbol{U}],\boldsymbol{V}_T\rangle_{\Gamma_R}| \\
	\leq&~C\sum_{m=0}^{\infty}\sum_{n=-\infty}^{\infty}(1+n^2+m^2)^{-1/2}\bigg(|U_{\theta,s,nm}||V_{\theta,s,nm}|+|U_{z,c,nm}||V_{z,c,nm}| \\
	&+\left|\frac{\mathrm{i}n}{R}U_{z,c,nm}-\frac{m\pi}{Z}U_{\theta,s,nm}\right|\left|\frac{m\pi}{Z}V_{\theta,s,nm}-\frac{\mathrm{i}n}{R}V_{z,c,nm}\right|\bigg) \\
	\leq&~C\|\boldsymbol{\nu}\times\boldsymbol{U}\|_{\boldsymbol{H}^{-1/2}(\mathrm{Div},\Gamma_R)}\|\boldsymbol{V}_T\|_{\boldsymbol{H}^{-1/2}(\mathrm{Curl},\Gamma_R)}.
\end{align*}
Then the proof is complete by noting that
\begin{align*}
	\|\mathscr{T}[\boldsymbol{\nu}\times\boldsymbol{U}]\|_{\boldsymbol{H}^{-1/2}(\mathrm{Div},\Gamma_R)}
	&=\sup_{\mathbf{0}\neq\boldsymbol{V}_T\in\boldsymbol{H}^{-1/2}(\mathrm{Curl},\Gamma_R)}\frac{|\langle\mathscr{T}[\boldsymbol{\nu}\times\boldsymbol{U}],\boldsymbol{V}_T\rangle_{\Gamma_R}|}{\|\boldsymbol{V}_T\|_{\boldsymbol{H}^{-1/2}(\mathrm{Curl},\Gamma_R)}} \\
	&\leq C\|\boldsymbol{\nu}\times\boldsymbol{U}\|_{\boldsymbol{H}^{-1/2}(\mathrm{Div},\Gamma_R)}.
\end{align*}
\end{proof}

\begin{lemma}\label{imaginary part of boundary integral}
	For any $\bm\nu\times\boldsymbol{U}\in \boldsymbol{H}^{-1/2}(\mathrm{Div},\Gamma_R)$, it holds that
	\begin{equation*}
		\mathrm{Im}\langle\mathscr{T}[\boldsymbol{\nu}\times\boldsymbol{U}],\boldsymbol{U}_T\rangle_{\Gamma_R}\leq0.
	\end{equation*}
\end{lemma}
\begin{proof}
It follows from the definition of Calder\'{o}n operator $ \mathscr{T} $ that
	\begin{align*}
		\langle\mathscr{T}[\boldsymbol{\nu}\times\boldsymbol{U}],\boldsymbol{\nu}\times(\boldsymbol{U}\times\boldsymbol{\nu})\rangle_{\Gamma_R}
		=\pi RZ\sum_{n=-\infty}^{\infty}\sum_{m=0}^{\infty}[\overline{U}_{\theta,s,nm},\overline{U}_{z,c,nm}]\widetilde{W}_{nm}
	\begin{bmatrix}
		U_{\theta,s,nm} \\
		U_{z,c,nm}
	\end{bmatrix}.
	\end{align*}
where $\widetilde{W}_{nm}^{11}=W_{nm}^{11}$, $\widetilde{W}_{nm}^{12}=-W_{nm}^{12}$, $\widetilde{W}_{nm}^{21}=W_{nm}^{21}$ and $\widetilde{W}_{nm}^{22}=-W_{nm}^{22}$. For $ m<kZ/\pi $, it follows from $ \mathrm{Im}[H_n^{(1)\prime}(k_mR)/H_n^{(1)}(k_mR)]>0 $~\cite{MR2988887} that
\begin{equation*}
	\mathrm{Im}[H_n^{(1)}(k_mR)/H_n^{(1)\prime}(k_mR)]<0.
\end{equation*}
This gives
\begin{align*}
	\mathrm{Im}[\widetilde{W}_{nm}^{11}]
	=&~\mathrm{Im}\left[\frac{k_mH_n^{(1)}(k_mR)}{H_n^{(1)\prime}(k_mR)}\right]<0, \\
	\det(\mathrm{Im}[\widetilde{W}_{nm}])
	=&-k^2\mathrm{Im}\left[\frac{H_n^{(1)}(k_mR)}{H_n^{(1)\prime}(k_mR)}\right]\mathrm{Im}\left[\frac{H_n^{(1)\prime}(k_mR)}{H_n^{(1)}(k_mR)}\right]>0.
\end{align*}
Hence, $ \mathrm{Im}[\widetilde{W}_{nm}] $ is negative-definite for $ m<kZ/\pi $. For $ m>kZ/\pi $, since $ K_n(t) $ are real for all $ n $ and $ t>0 $, we have
\begin{align*}
	\widetilde{W}_{nm}^*=\widetilde{W}_{nm}.
\end{align*}
The proof is complete.
\end{proof}

The uniqueness of the variational problem (\ref{variational problem}) is given in the following lemma.

\begin{lemma}\label{uniqueness result}
The variational problem (\ref{variational problem}) has at most one solution.
\end{lemma}
\begin{proof}
Let $ \boldsymbol{f}=\mathbf{0} $. Taking the imaginary part of (\ref{variational problem}) with $ \boldsymbol{V}=\boldsymbol{E} $ gives that
	\begin{equation*}
		\langle\eta\E_T,\E_T\rangle_{\partial D}
		-\mathrm{Im}\langle\mathscr{T}[\boldsymbol{n}\times\boldsymbol{E}],\boldsymbol{E}_T\rangle_{\Gamma_R}
		=0,
	\end{equation*}
which implies from Lemma~\ref{imaginary part of boundary integral} and the boundary condition in (\ref{model:E}) that
    \begin{equation*}
        \bm{\nu}\times\nabla\times\bm{E}
        =\bm{E}_T
        =\mathbf{0} \quad \text{on}~\partial D.
    \end{equation*} 
Applying again the unique continuation result \cite[Theorem 4.13]{MR2059447} completes the proof.
\end{proof}

Since $ \boldsymbol{X} $ is not compactly embedded into $ \boldsymbol{L}^2(\Omega_R) $, we next present a Helmholtz decomposition and a compactness result which plays a crucial role in establishing the well-posedness of the variational problem (\ref{variational problem}). Denote
\begin{align*}
	\boldsymbol{X}_0
	&=\left\{\boldsymbol{u}\in\boldsymbol{X}:\nabla\cdot\boldsymbol{u}=0~\text{in}~\Omega_R,~k^2\boldsymbol{n}\cdot\boldsymbol{u}=-\nabla_{\Gamma_R}\cdot\mathscr{T}[\boldsymbol{n}\times\boldsymbol{u}]~\text{on}~\Gamma_R\right\}, \\
	S
	&=\left\{u \in H^1(\Omega_R):u=0~\text{on}~\Sigma_R,~u=\text{constant}~\text{on}~\partial D\right\}.
\end{align*}
Since $ \gamma_t\nabla p=\boldsymbol{\nu}\times\nabla p=\mathbf{0} $ holds for $ p \in S $ on $ \Sigma_R $, it follows that $ \nabla S\subset\boldsymbol{X} $. Thereinafter consider $ \boldsymbol{E}=\boldsymbol{u}+\nabla p $ for $ \boldsymbol{u}\in\boldsymbol{X}_0 $ and $ p \in S $.

We first seek $ p \in S $ such that
\begin{equation}\label{variational problem for p}
	a(\nabla p,\nabla q)=b(\nabla q) \quad \forall q \in S.
\end{equation}

\begin{lemma}\label{well-posedness of variational problem for p}
There exists a unique solution $p\in S$ to the variational problem (\ref{variational problem for p}).
\end{lemma}
\begin{proof}
	A simple calculation yields
	\begin{equation*}
		a(\nabla p,\nabla q)=-k^2(\nabla p,\nabla q)+\langle\mathscr{T}[\boldsymbol{\nu}\times\nabla p],\nabla_{\Gamma_R}q\rangle_{\Gamma_R}.
	\end{equation*}
Denote $ B(p,q)=-a(\nabla p,\nabla q) $. Applying the Cauchy-Schwarz inequality and Lemma \ref{continuity of T}, we obtain that
\begin{align*}
	&|B(p,q)| \\
	&\leq k^2\|\nabla p\|_{\boldsymbol{L}^2(\Omega_R)}\|\nabla q\|_{\boldsymbol{L}^2(\Omega_R)}+\|\mathscr{T}[\boldsymbol{\nu}\times\nabla p]\|_{\boldsymbol{H}^{-1/2}(\mathrm{Div},\Gamma_R)}\|\nabla_{\Gamma_R}q\|_{\boldsymbol{H}^{-1/2}(\mathrm{Curl},\Gamma_R)} \\
	&\leq k^2\|p\|_{H^1(\Omega_R)}\|q\|_{H^1(\Omega_R)}+C\|\nabla p\|_{\boldsymbol{H}(\mathrm{curl},\Omega_R)}\|\nabla q\|_{\boldsymbol{H}(\mathrm{curl},\Omega_R)} \\
	&\leq C\|p\|_{H^1(\Omega_R)}\|q\|_{H^1(\Omega_R)}.
\end{align*}
For $ p\in S$ which admits a Fourier series
\begin{equation*}
	p|_{\Gamma_R}
	=\sum_{m=0}^{\infty}\sum_{n=-\infty}^{\infty}p_{nm}\sin(m\pi z/Z)e^{\mathrm{i}n\theta},
\end{equation*}
Hence,
\begin{align*}
	\nabla_{\Gamma_R}p
	&=\sum_{m=0}^{\infty}\sum_{n=-\infty}^{\infty}p_{nm}\left(\frac{\mathrm{i}n}{R}\boldsymbol{S}_{\theta,nm}+\frac{m\pi}{Z}\boldsymbol{C}_{z,nm}\right), \\
	\boldsymbol{n}\times\nabla p
	&=\sum_{m=0}^{\infty}\sum_{n=-\infty}^{\infty}p_{nm}\left(-\frac{m\pi}{Z}\boldsymbol{C}_{\theta,nm}+\frac{\mathrm{i}n}{R}\boldsymbol{S}_{z,nm}\right),
\end{align*}
which gives
\begin{align*}
	&\langle\mathscr{T}[\boldsymbol{n}\times\nabla p],\nabla_{\Gamma_R}p\rangle_{\Gamma_R} \\
	=&~\pi RZ\sum_{m=0}^{\infty}\sum_{n=-\infty}^{\infty}|p_{nm}|^2\left(\widetilde{W}_{nm}^{11}\frac{n^2}{R^2}-\widetilde{W}_{nm}^{12}\frac{\mathrm{i}nm\pi}{RZ}+\widetilde{W}_{nm}^{21}\frac{\mathrm{i}nm\pi}{RZ}+\widetilde{W}_{nm}^{22}\frac{m^2\pi^2}{Z^2}\right) \\
	:=&~\pi RZ\sum_{m=0}^{\infty}\sum_{n=-\infty}^{\infty}R_{nm}|p_{nm}|^2.
\end{align*}
For $ m<kZ/\pi $, we know for $ |n|>N $ with $N>0$ being sufficiently large that $R_{nm}=-|n|R^{-1}k^2+O(n^{-1})$ which implies $\mathrm{Re}[R_{nm}]< 0 $. For $ m>kZ/\pi $, it can be derived via straightforward calculations that, for $ |n|\in\mathbb{N}^+$,
\begin{align*}
	&R_{nm} \\
    =&~\frac{k^2K_n(|k_m|R)}{|k_m|K_n^{\prime}(|k_m|R)}\left(\frac{n^2}{R^2}+\frac{m^2\pi^2}{Z^2}\left(\frac{K_{|n|-1}(|k_m|R)}{K_{|n|}(|k_m|R)}\right)^2+2\frac{K_{|n|-1}(|k_m|R)}{K_{|n|}(|k_m|R)}\frac{|n|m^2\pi^2}{|k_m|RZ^2}\right).
\end{align*}
Combined with $ R_{0m}=\frac{m^2\pi^2}{Z^2}k^2\frac{K_0^{\prime}(|k_m|R)}{|k_m|K_0(|k_m|R)}$ and the fact that $K_n(|k_m|R)>0$ and $K_n'(|k_m|R)<0$ for $|n|\in\mathbb{N}$, we conclude $ \mathrm{Re}[R_{nm}]<0 $ for $ m>kZ/\pi $.

Define
\begin{align*}
	&\langle\mathscr{T}[\boldsymbol{\nu}\times\nabla p],\nabla_{\Gamma_R}p\rangle_{\Gamma_R} \\
	=&~\int_{\Gamma_R} \left(\sum_{m<kZ/\pi}\sum_{|n| \leq N}
	+\sum_{m<kZ/\pi}\sum_{|n|>N}
	+\sum_{m>kZ/\pi}\sum_{n=-\infty}^{\infty}\right) R_{nm}|p_{nm}|^2\mathrm{d}s \\
	:=&~\langle(\mathscr{T}_1^-+\mathscr{T}_2^-+\mathscr{T}^+)[\boldsymbol{\nu}\times\nabla p],\nabla_{\Gamma_R}p\rangle_{\Gamma_R}.
\end{align*}
Then we have that $ \mathscr{T}_1^- $ is a compact operator, and $ \mathscr{T}_2^-+\mathscr{T}^+ $ is non-positive over $ \boldsymbol{H}^{-1/2}(\mathrm{Div},\Gamma_R) $. Utilizing the Poincar\'{e} inequality, we obtain
\begin{align*}
	\mathrm{Re}[B(p,p)]
	&\geq k^2\|\nabla p\|_{\boldsymbol{L}^2(\Omega_R)}^2-C_N\|p\|_{L^2(\Gamma_R)}^2 \\
	&\geq C\|p\|_{H^1(\Omega_R)}^2-C_N\delta\|p\|_{H^1(\Omega_R)}^2-C(\delta)\|p\|_{L^2(\Omega_R)}^2 \\
	&\geq (C-C_N\delta)\|p\|_{H^1(\Omega_R)}^2-C(\delta)\|p\|_{L^2(\Omega_R)}^2.
\end{align*}
Letting $ \delta $ be sufficiently small such that $ C-C_N\delta>0 $, the Garding's inequality holds.

Finally, the proof of this lemma can be completed by using the Fredholm alternative if uniqueness holds. Taking the imaginary part of (\ref{variational problem for p}) with $ q=p $, we have
\begin{equation*}
	\mathrm{Im}\langle\mathscr{T}[\boldsymbol{\nu}\times\nabla p],\nabla_{\Gamma_R}p\rangle_{\Gamma_R}=0,
\end{equation*}
which implies from Lemma \ref{imaginary part of boundary integral} that $ p_{nm}=0 $ holds for any $|n|\in\mathbb{N}$ and $ m<kZ/\pi $. Taking the real part of (\ref{variational problem for p}), we obtain that
\begin{align*}
	0
	&=k^2\int_{\Omega_R}|\nabla p|^2\mathrm{d}\boldsymbol{x}-\mathrm{Re}\langle\mathscr{T}[\boldsymbol{\nu}\times\nabla p],\nabla_{\Gamma_R}p\rangle_{\Gamma_R} \\
	&=k^2\int_{\Omega_R}|\nabla p|^2\mathrm{d}\boldsymbol{x}-\pi RZ\sum_{m>kZ/\pi}\sum_{n=-\infty}^{\infty}R_{nm}|p_{nm}|^2,
\end{align*}
which combined with $R_{nm}=\mathrm{Re}[R_{nm}]<0 $, gives $ p_{nm}=0 $ for $|n|\in\mathbb{N}$ and $ m>kZ/\pi $. Thus,
\begin{align*}
	k^2\int_{\Omega_R}|\nabla p|^2\mathrm{d}\boldsymbol{x}
	=\langle\mathscr{T}[\boldsymbol{\nu}\times\nabla p],\nabla_{\Gamma_R}p\rangle_{\Gamma_R}
	=0,
\end{align*}
which shows $ \nabla p|_{\Omega_R}=\mathbf{0} $. We conclude via the Poincar\'{e} inequality that $ p=0 $. The proof is complete.
\end{proof}

The Helmholtz decomposition and the related compactness results are given as follows. The proof is analogous to \cite[Lemma 3.25 and Lemma 3.26]{MR4385553} or \cite[Lemma 10.3 and Lemma 10.4]{MR2059447} and thus, is omitted.

\begin{lemma}
	The spaces $ \boldsymbol{X}_0 $ and $ \nabla S $ are closed subspaces of $ \boldsymbol{X} $, which is a direct sum of $ \boldsymbol{X}_0 $ and $ \nabla S $, i.e.,
	\begin{equation*}
		\boldsymbol{X}
		=\boldsymbol{X}_0\oplus\nabla S.
	\end{equation*}
\end{lemma}

\begin{lemma}
	The space $ \boldsymbol{X}_0 $ is compactly imbedded in $ \boldsymbol{L}^2(\Omega_R) $.
\end{lemma}

Now, we are ready to establish the well-posedness and the inf-sup condition for the variational problem (\ref{variational problem}).

\begin{theorem}
	The variational problem (\ref{variational problem}) has a unique solution $ \boldsymbol{E}\in\boldsymbol{X} $, which admits a Helmholtz decomposition $ \boldsymbol{E}=\boldsymbol{u}+\nabla p $ with $ \boldsymbol{u}\in\boldsymbol{X}_0 $ and $ p \in S $. Moreover, the inf-sup condition
	\begin{equation*}
		\sup_{\mathbf{0}\neq\boldsymbol{V}\in\boldsymbol{X}}\frac{|a(\boldsymbol{E},\boldsymbol{V})|}{\|\boldsymbol{V}\|_{\boldsymbol{X}}}
		\geq C_{\mathrm{inf}}\|\boldsymbol{E}\|_{\boldsymbol{X}} ,\quad \forall \boldsymbol{E}\in\boldsymbol{X},
	\end{equation*}
holds where the constant $ C_{\mathrm{inf}}>0 $ depends only on $ k,R,Z $.
\end{theorem}
\begin{proof}
	Given the Helmholtz decomposition for $\bm E,\bm V\in\bm X$,
	\begin{equation*}
		\boldsymbol{E}=\boldsymbol{u}+\nabla p, \quad \boldsymbol{V}=\boldsymbol{v}+\nabla q, \quad \bm{u},\bm{v}\in\boldsymbol{X}_0, p,q \in S.
	\end{equation*}
we have
\begin{align*}
	a(\boldsymbol{u},\nabla q)
	&=-k^2\int_{\Omega_R}\boldsymbol{u}\cdot\nabla\overline{q}\mathrm{d}\boldsymbol{x}+\int_{\Gamma_R}\mathscr{T}[\boldsymbol{\nu}\times\boldsymbol{u}]\cdot\nabla_{\Gamma_R}\overline{q}\mathrm{d}s \\
	&=k^2\int_{\Omega_R}\nabla\cdot\boldsymbol{u}\overline{q}\mathrm{d}\boldsymbol{x}-\int_{\Gamma_R}(k^2\boldsymbol{\nu}\cdot\boldsymbol{u}+\nabla_{\Gamma_R}\cdot\mathscr{T}[\boldsymbol{\nu}\times\boldsymbol{u}])\overline{q}\mathrm{d}s \\
	&=0.
\end{align*}
The variational problem (\ref{variational problem}) can be decomposed into the form
\begin{equation}\label{mixed variational problem}
	a(\boldsymbol{u},\boldsymbol{v})+a(\nabla p,\boldsymbol{v})+a(\nabla p,\nabla q)=b(\boldsymbol{v})+b(\nabla q), \quad \forall \boldsymbol{v}\in\boldsymbol{X}_0, q \in S.
\end{equation}
First, we can determine $ p \in S $ by the variational problem
\begin{equation*}
	a(\nabla p,\nabla q)=b(\nabla q), \quad \forall q \in S.
\end{equation*}
which, according to Lemma \ref{well-posedness of variational problem for p}, has a unique solution in $S$.

Then the variational problem (\ref{mixed variational problem}) can be reformulated as: find $ \boldsymbol{u}\in\boldsymbol{X}_0 $ such that
\begin{equation}\label{variational problem for u}
	a(\boldsymbol{u},\boldsymbol{v})=b(\boldsymbol{v})-a(\nabla p,\boldsymbol{v}), \quad \forall \boldsymbol{v}\in\boldsymbol{X}_0.
\end{equation}
The continuity of the sesquilinear form $ a(\cdot,\cdot) $ follows from the Cauchy-Schwarz inequality and Lemma \ref{continuity of T} that
\begin{align*}
	|a(\boldsymbol{u},\boldsymbol{v})|
	&\leq C_1\|\boldsymbol{u}\|_{\boldsymbol{X}}\|\boldsymbol{v}\|_{\boldsymbol{X}}+\|\mathscr{T}[\boldsymbol{\nu}\times\boldsymbol{u}]\|_{\boldsymbol{H}^{-1/2}(\mathrm{Div},\Gamma_R)}\|\boldsymbol{v}_T\|_{\boldsymbol{H}^{-1/2}(\mathrm{Curl},\Gamma_R)} \\
	&\leq C_1\|\boldsymbol{u}\|_{\boldsymbol{X}}\|\boldsymbol{v}\|_{\boldsymbol{X}}+C_2\|\boldsymbol{\nu}\times\boldsymbol{u}\|_{\boldsymbol{H}^{-1/2}(\mathrm{Div},\Gamma_R)}\|\boldsymbol{v}_T\|_{\boldsymbol{H}^{-1/2}(\mathrm{Curl},\Gamma_R)} \\
	& \leq C\|\boldsymbol{u}\|_{\boldsymbol{X}}\|\boldsymbol{v}\|_{\boldsymbol{X}}.
\end{align*}
We next split the sesquilinear form $ a(\cdot,\cdot) $ into $ a=a_1+a_2 $, where
\begin{align*}
	a_1(\boldsymbol{u},\boldsymbol{v})
	&=(\nabla\times\boldsymbol{u},\nabla\times\boldsymbol{v})+k^2(\boldsymbol{u},\boldsymbol{v})-\mathrm{i}\langle\eta\boldsymbol{u}_T,\boldsymbol{v}_T\rangle_{\partial D}+\langle\mathscr{T}_B[\bm{\nu}\times\bm{u}],\bm{v}_T\rangle_{\Gamma_R}, \\
	a_2(\boldsymbol{u},\boldsymbol{v})
	&=-2k^2(\boldsymbol{u},\boldsymbol{v})+\langle(\mathscr{T}_A-\mathscr{T}_B)[\boldsymbol{\nu}\times\boldsymbol{u}],\boldsymbol{v}_T\rangle_{\Gamma_R}+\langle\mathscr{T}_C[\gamma_t\bm{u}],\bm{v}_T\rangle_{\Gamma_R}.
\end{align*}
Here the Calderon operator $ \mathscr{T} $ is split into
\begin{align*}
	&\mathscr{T}_A[\bm{\nu}\times\bm{u}] \\
	=&\sum_{m=0}^{\infty}\sum_{n=-\infty}^{\infty}\frac{\beta_mT_{\theta,nm}-\alpha_nT_{z,nm}}{|\alpha_n|^2+\beta_m^2}\left(\beta_m\bm{S}_{\theta,nm}+\alpha_n\bm{C}_{z,nm}\right), \\
	&\mathscr{T}_B[\bm{\nu}\times\bm{u}] \\
	=&\left\{
    \begin{aligned}
        &\sum_{m=0}^{\infty}\sum_{n=-\infty}^{\infty}|\alpha_n|^{-1}\left(\beta_mu_{\theta,c,nm}-\alpha_nu_{z,c,nm}\right)(\beta_m\bm{S}_{\theta,nm}+\alpha_n\bm{C}_{z,nm}) \quad m<k\pi/Z, \\
	&\sum_{m=0}^{\infty}\sum_{n=-\infty}^{\infty}-\frac{H_n^{(1)}(k_mR)}{k_mH_n^{(1)\prime}(k_mR)}\left(\beta_mu_{\theta,c,nm}-\alpha_nu_{z,c,nm}\right)(\beta_m\bm{S}_{\theta,nm}+\alpha_n\bm{C}_{z,nm}) \\
    &\quad m>k\pi/Z,
    \end{aligned}
    \right. \\
	&\mathscr{T}_C[\bm{\nu}\times\bm{u}] \\
	=&\sum_{m=0}^{\infty}\sum_{n=-\infty}^{\infty}\frac{-\alpha_nT_{\theta,nm}+\beta_mT_{z,nm}}{|\alpha_n|^2+\beta_m^2}\left(\alpha_n\bm{S}_{\theta,s,nm}+\beta_m\bm{C}_{z,nm}\right),
\end{align*}
where we rewrite $ \mathscr{T}[\bm{\nu}\times\bm{u}]=\sum_{m=0}^{\infty}\sum_{n=-\infty}^{\infty}(T_{\theta,nm}\bm{S}_{\theta,nm}+T_{z,nm}\bm{C}_{z,nm}) $ and denote $ \alpha_n=\mathrm{i}n/R $ and $ \beta_m=m\pi/Z $. In particular, $\nabla_{\Gamma_R}\cdot\mathscr{T}_A[\bm{\nu}\times\bm{u}]=\nabla_{\Gamma_R}\cdot\mathscr{T}_B[\bm{\nu}\times\bm{u}]=0$ and $\nabla_{\Gamma_R}\times\mathscr{T}_C[\bm{\nu}\times\bm{u}]=0$ for any $\bm{\nu}\times\bm{u}\in \bm{H}^{-1/2}(\mathrm{Div},\Gamma_R)$. Since it follows from \cite{MR2988887,MR2723248} that $ \mathrm{Re}[H_{n}^{(1)}(k_mR)/(k_mH_n^{(1)\prime}(k_mR))]\leq0 $, which implies $ \mathrm{Re}\langle\mathscr{T}_B[\bm{\nu\times\bm{u}}],\bm{u}_T\rangle_{\Gamma_R}\geq0 $ and $ \mathrm{Im}\langle\mathscr{T}_B[\bm{\nu\times\bm{u}}],\bm{u}_T\rangle_{\Gamma_R}=0 $, the coercivity of $ a_1(\cdot,\cdot) $ results. Next we prove that $ a_2(\cdot,\cdot) $ is a compact form. Noting that 
\begin{align*}
	&\beta_mT_{\theta,nm}-\alpha_nT_{z,nm} \\
	&=\frac{H_n^{(1)}(k_mR)}{k_mH_n^{(1)\prime}(k_mR)}\left(\alpha_n^2-\beta_m^2\right)\left(\beta_mu_{\theta,c,nm}-\alpha_nu_{z,c,nm}\right)+k^2\frac{H_n^{(1)}(k_mR)}{k_mH_n^{(1)\prime}(k_mR)}\beta_mu_{\theta,nm} \\
	&\quad +k^2\left(\frac{H_n^{(1)}(k_mR)}{k_mH_n^{(1)\prime}(k_mR)}\frac{\alpha_n^2}{k_m^2}+\frac{H_n^{(1)\prime}(k_mR)}{k_mH_n^{(1)}(k_mR)}\right)\alpha_nu_{z,c,nm}.
\end{align*}
By an analogous argument in Lemma \ref{continuity of T}, it gives that
\begin{align*}
	\left|\frac{H_n^{(1)}(k_mR)}{k_mH_n^{(1)\prime}(k_mR)}\frac{\alpha_n^2}{k_m^2}+\frac{H_n^{(1)\prime}(k_mR)}{k_mH_n^{(1)}(k_mR)}\right|
	\leq C(1+|n|^2+|m|^2)^{-1/2},
\end{align*}
which implies that
\begin{align*}
	\|(\mathscr{T}_A-\mathscr{T}_B)[\bm{\nu}\times\bm{u}]\|_{\bm{H}_t^{1/2}(\Gamma_R)}
	\leq C\|\bm{\nu}\times\bm{u}\|_{\bm{H}^{-1/2}(\mathrm{Div},\Gamma_R)}.
\end{align*}
Thus, $\mathscr{T}_A-\mathscr{T}_B:\bm{H}^{-1/2}(\mathrm{Div},\Gamma_R)\to\bm{H}_t^{1/2}(\Gamma_R) $ is well-defined, which implies that $\mathscr{T}_A-\mathscr{T}_B$ is compact from $ \bm{H}^{-1/2}(\mathrm{Div},\Gamma_R) $ into itself. The compactness of $(\mathscr{T}_A-\mathscr{T}_B)\circ\gamma_t $ from $ \boldsymbol{X}_0 $ to $ \boldsymbol{H}^{-1/2}(\mathrm{Div},\Gamma_R) $ is concluded by the trace theorem. From the definition of norm on $ \boldsymbol{H}^{-1/2}(\mathrm{Div},\Gamma_R) $ and the trace theorem of $ \bm{H}(\mathrm{Div},\Omega_R) $, we have for $ \boldsymbol{u}\in\boldsymbol{X}_0 $ that
\begin{align*}
	\|(\mathscr{T}_C\circ\gamma_t)[\boldsymbol{u}]\|_{\boldsymbol{H}_t^{-1/2}(\mathrm{Div},\Gamma_R)}
	&=\|\mathscr{T}_C[\boldsymbol{\nu}\times\boldsymbol{u}]\|_{\boldsymbol{H}^{-1/2}(\mathrm{Div},\Gamma_R)} \\
	&\leq C\|\nabla_{\Gamma_R}\cdot\mathscr{T}[\boldsymbol{\nu}\times\boldsymbol{u}]\|_{H^{-1/2}(\Gamma_R)} \\
	&= Ck^2\|\boldsymbol{\nu}\cdot\boldsymbol{u}\|_{H^{-1/2}(\Gamma_R)} \\
	&\leq C\left(\|\boldsymbol{u}\|_{\boldsymbol{L}^2(\Omega_R)}^2+\|\nabla\cdot\boldsymbol{u}\|_{L^2(\Omega_R)}^2\right)^{1/2} \\
	&= C\|\boldsymbol{u}\|_{\boldsymbol{L}^2(\Omega_R)},
\end{align*}
which gives that $ \mathscr{T}_C\circ\gamma_t:\bm{X}_0\to\bm{H}^{-1/2}(\mathrm{Div},\Gamma_R) $ is bounded. The compactness of $ \boldsymbol{X}_0 $ in $ \boldsymbol{L}^2(\Omega_R) $ yields the assertion. Define an operator$ \mathscr{K}:\boldsymbol{X}_0\to\boldsymbol{X}_0 $ by
\begin{equation*}
	a_2(\boldsymbol{u},\boldsymbol{v})=a_1(\mathscr{K}[\boldsymbol{u}],\boldsymbol{v}) \quad \forall \boldsymbol{u},\boldsymbol{v}\in\boldsymbol{X}_0.
\end{equation*}
By the compact injection $ \boldsymbol{X}_0\hookrightarrow\hookrightarrow\boldsymbol{L}^2(\Omega_R) $ and the compactness of operator $ (\mathscr{T}_A-\mathscr{T}_B+\mathscr{T}_C)\circ\gamma_t $, we have that $ \mathscr{K} $ is a compact operator.

Rewrite the variational problem (\ref{variational problem for u}) into an equivalent operator equation
\begin{equation}\label{operator equation}
	\boldsymbol{u}+\mathscr{K}[\boldsymbol{u}]=\boldsymbol{\psi},
\end{equation}
where $ \boldsymbol{\psi}\in\boldsymbol{X}_0 $ is the unique solution to the problem
\begin{equation*}
	a_1(\boldsymbol{\psi},\boldsymbol{v})=b(\boldsymbol{v}) \quad \forall \boldsymbol{v}\in\boldsymbol{X}_0.
\end{equation*}
Since $ \mathscr{K}:\boldsymbol{X}_0\to\boldsymbol{X}_0 $ is a compact operator, (\ref{operator equation}) is a Fredholm equation of the second kind. By Lemma \ref{uniqueness result} and the Fredholm alternative theorem, there exists a unique solution $ \boldsymbol{u}\in\boldsymbol{X}_0 $ to the variational problem (\ref{variational problem for u}). Let $ \mathscr{I} $ denote the identity operator on $ \boldsymbol{X}_0 $. We conclude that $ (\mathscr{I}+\mathscr{K})^{-1}:\boldsymbol{X}_0\to\boldsymbol{X}_0 $ exists and is a continuous operator, such that $ \boldsymbol{u}=(\mathscr{I}+\mathscr{K})^{-1}\boldsymbol{\psi} $.

Since the variational problem (\ref{variational problem}) is well-posed, we have from the general theory in \cite{MR0421106} that there exists a constant $ C_{\mathrm{inf}}>0 $ such that the inf-sup condition
\begin{equation*}
	\sup_{\mathbf{0}\neq\boldsymbol{V}\in\boldsymbol{X}}\frac{|a(\boldsymbol{E},\boldsymbol{V})|}{\|\boldsymbol{V}\|_{\boldsymbol{X}}}
	\geq C_{\mathrm{inf}}\|\boldsymbol{E}\|_{\boldsymbol{X}}, \quad \forall \boldsymbol{E}\in\boldsymbol{X}
\end{equation*}
holds. The proof is complete.
\end{proof}

\section{Uniqueness result of the inverse obstacle problem}\label{sec 4}
In this section, we turn to study the inverse problem of determining the obstacle $ D $ form the measured scattered field $ \boldsymbol{E} $ on the boundary $ \Gamma_R $ with 
\begin{equation*}
	\boldsymbol{f}
	=-\mathscr{B}[\boldsymbol{E}^\mathrm{inc}]
	:=-\boldsymbol{\nu}\times\nabla\times\boldsymbol{E}^\mathrm{inc}(\boldsymbol{x};\boldsymbol{y},\boldsymbol{p})+\mathrm{i}\eta\boldsymbol{E}_T^\mathrm{inc}(\boldsymbol{x};\boldsymbol{y},\boldsymbol{p}),
\end{equation*}
where $ \boldsymbol{y}\in\Gamma_{R^\prime} $ with $ R \geq R^\prime $. In particular, the incident wave $\boldsymbol{E}^\mathrm{inc}$ is set to be a point source at $ \boldsymbol{y}\in\Omega $ with polarization $ \boldsymbol{p}\neq\mathbf{0} $ is given by
\begin{equation}
	\boldsymbol{E}^\mathrm{inc}(\boldsymbol{x};\boldsymbol{y},\boldsymbol{p})=\mathbb{G}(\boldsymbol{x},\boldsymbol{y})\boldsymbol{p},
\end{equation}
where the $ \mathbb{G} $ denotes the dyadic Green's function satisfying
\begin{equation}
	\left\{
	\begin{aligned}
		&\nabla_{\boldsymbol{x}}\times\nabla_{\boldsymbol{x}}\times\mathbb{G}-k^2\mathbb{G}=\delta_{\boldsymbol{y}}(\boldsymbol{x})\mathbb{I} \quad && \text{in}~\R_Z^3, \\
		&\boldsymbol{\nu}\times\mathbb{G}=\mathbb{O} \quad && \text{on}~\Gamma^\pm,
	\end{aligned}
	\right.
\end{equation}
together with the radiation condition (\ref{radiation1})-(\ref{radiation2}). Here, $ \mathbb{I} $ is the $3\times 3$ identity matrix, $ \mathbb{O} $ is the $3\times 3$ zero matrix, and $\boldsymbol{a}\times\mathbb{A}$, $ \nabla\times\mathbb{A} $ denote the matrices whose $ l $-th columns are given by $ \boldsymbol{a}\times\boldsymbol{a}_l $ and $ \nabla\times\boldsymbol{a}_l $, respectively, for any vector function $ \boldsymbol{a} $ and any matrix function $ \mathbb{A} $ with columns $ \boldsymbol{a}_l $.

In order to prove the uniqueness result for the inverse problem, we start by stating two lemmas. The first gives a representation formula of the solution.
\begin{lemma}[Green's representation formula]\label{Green's representation formula}
Let $ \E $ be the unique solution to the problem (\ref{model:E}). Then
\begin{equation*}
	\E(\boldsymbol{x})
	=\int_{\partial D}\mathbb{G}^\top(\boldsymbol{x},\boldsymbol{y})\cdot\gamma_t(\nabla\times\E(\boldsymbol{y}))+(\nabla_{\boldsymbol{y}}\times\mathbb{G}(\boldsymbol{x},\boldsymbol{y}))^\top\cdot\gamma_t\E(\boldsymbol{y})\mathrm{d}s(\boldsymbol{y})
\end{equation*}
\end{lemma}
\begin{proof}
The proof follows analogous to \cite[Theorem 12.2]{MR2059447} or \cite[Lemma 3.3.2]{MR3439069}.
\end{proof}

The next result states a reciprocity relation for the considered problem.
\begin{lemma}[Reciprocity relation]\label{Reciprocity relation}
	Denote by $ \E(\boldsymbol{x};\boldsymbol{y},\boldsymbol{p}) $ the solution to the scattering problem (\ref{model:E}) with \begin{equation*}
		\boldsymbol{f}
		=-\mathscr{B}[\mathbb{G}(\boldsymbol{x},\boldsymbol{y})\boldsymbol{p}].
	\end{equation*}
For any $ \boldsymbol{x},\boldsymbol{z}\in\Omega $ and polarizations $ \boldsymbol{p},\boldsymbol{q} $, it holds that
\begin{equation*}
	\boldsymbol{q}\cdot\E(\boldsymbol{x};\boldsymbol{z},\boldsymbol{p})
	=\boldsymbol{p}\cdot\E(\boldsymbol{z};\boldsymbol{x},\boldsymbol{p}).
\end{equation*}
\end{lemma}
\begin{proof}
	By using the representation formula in Lemma \ref{Green's representation formula} and identities \cite[Appendix B]{MR3439069}, we obtain
	\begin{equation*}
		\boldsymbol{E}(\boldsymbol{x};\boldsymbol{z},\boldsymbol{p})
		=\int_{\partial D}-\boldsymbol{E}(\boldsymbol{y};\boldsymbol{z},\boldsymbol{p})\cdot\gamma_t(\nabla_{\boldsymbol{y}}\times\mathbb{G}(\boldsymbol{x},\boldsymbol{y}))+\gamma_t(\nabla\times\boldsymbol{E}(\boldsymbol{y};\boldsymbol{z},\boldsymbol{p}))\cdot\mathbb{G}(\boldsymbol{x},\boldsymbol{y})\mathrm{d}s(\boldsymbol{y}).
	\end{equation*}
which gives
\begin{equation}
\label{RR1}
	\boldsymbol{q}\cdot\boldsymbol{E}(\boldsymbol{x};\boldsymbol{z},\boldsymbol{p})
	=\int_{\partial D}-\boldsymbol{E}(\boldsymbol{y};\boldsymbol{z},\boldsymbol{p})\cdot\gamma_t(\nabla_{\boldsymbol{y}}\times\mathbb{G}(\boldsymbol{x},\boldsymbol{y})\boldsymbol{q})+\gamma_t(\nabla\times\boldsymbol{E}(\boldsymbol{y};\boldsymbol{z},\boldsymbol{p}))\cdot\mathbb{G}(\boldsymbol{x},\boldsymbol{y})\boldsymbol{q}\mathrm{d}s(\boldsymbol{y}).
\end{equation}
Analogously, 
\begin{equation}
\label{RR2}
	\boldsymbol{p}\cdot\boldsymbol{E}(\boldsymbol{z};\boldsymbol{x},\boldsymbol{q})
	=\int_{\partial D}-\boldsymbol{E}(\boldsymbol{y};\boldsymbol{x},\boldsymbol{q})\cdot\gamma_t(\nabla_{\boldsymbol{y}}\times\mathbb{G}(\boldsymbol{z},\boldsymbol{y})\boldsymbol{p})+\gamma_t(\nabla\times\boldsymbol{E}(\boldsymbol{y};\boldsymbol{x},\boldsymbol{q}))\cdot\mathbb{G}(\boldsymbol{z},\boldsymbol{y})\boldsymbol{p}\mathrm{d}s(\boldsymbol{y}).
\end{equation}

By Green's second identity, we have for $ \boldsymbol{x},\boldsymbol{z}\in\Omega $ that
\begin{align}
\label{RR3}
	&\int_{\partial D}-\mathbb{G}(\boldsymbol{y},\boldsymbol{z})\boldsymbol{p}\cdot(\boldsymbol{\nu}\times\nabla_{\boldsymbol{y}}\times\mathbb{G}(\boldsymbol{x},\boldsymbol{y})\boldsymbol{q})+(\boldsymbol{\nu}\times\nabla\times\mathbb{G}(\boldsymbol{y},\boldsymbol{z})\boldsymbol{p})\cdot\mathbb{G}(\boldsymbol{x},\boldsymbol{y})\boldsymbol{q}\mathrm{d}s(\boldsymbol{y}) \nonumber\\
	=&\,\int_{\partial D}\boldsymbol{\nu}\cdot\left(\mathbb{G}(\boldsymbol{y},\boldsymbol{z})\boldsymbol{p}\times\nabla_{\boldsymbol{y}}\times\mathbb{G}(\boldsymbol{x},\boldsymbol{y})\boldsymbol{q}+(\nabla\times\mathbb{G}(\boldsymbol{y},\boldsymbol{z})\boldsymbol{p})\times\mathbb{G}(\boldsymbol{x},\boldsymbol{y})\boldsymbol{q}\right)\mathrm{d}s(\boldsymbol{y}) \nonumber\\
	=&\,\int_{D}(\nabla\times\nabla\times\mathbb{G}(\boldsymbol{y},\boldsymbol{z})\boldsymbol{p})\cdot\mathbb{G}(\boldsymbol{x},\boldsymbol{y})\boldsymbol{q}-\mathbb{G}(\boldsymbol{y},\boldsymbol{z})\boldsymbol{p}\cdot(\nabla_{\boldsymbol{y}}\times\nabla_{\boldsymbol{y}}\times\mathbb{G}(\boldsymbol{x},\boldsymbol{y})\boldsymbol{q})\mathrm{d}\boldsymbol{y} \nonumber\\
	=&\,\int_{D}(\nabla\times\nabla\times\mathbb{G}(\boldsymbol{y},\boldsymbol{z})\boldsymbol{p}-k^2\mathbb{G}(\boldsymbol{y},\boldsymbol{z})\boldsymbol{p})\cdot\mathbb{G}(\boldsymbol{x},\boldsymbol{y})\boldsymbol{q}\mathrm{d}\boldsymbol{y} \nonumber\\
	&-\int_{D}\mathbb{G}(\boldsymbol{y},\boldsymbol{z})\boldsymbol{p}\cdot(\nabla_{\boldsymbol{y}}\times\nabla_{\boldsymbol{y}}\times\mathbb{G}(\boldsymbol{x},\boldsymbol{y})\boldsymbol{q}-k^2\mathbb{G}(\boldsymbol{x},\boldsymbol{y})\boldsymbol{q})\mathrm{d}\boldsymbol{y} \nonumber\\
	=&\,0.
\end{align}
Using the Green's formula over $ \Omega_R\backslash\overline{D} $ and series expansions (\ref{series expansion of Er})-(\ref{series expansion of Ez}), we can also get
\begin{equation}\label{RR4}
	\begin{aligned}
		&\int_{\partial D}-\boldsymbol{E}(\boldsymbol{y};\boldsymbol{z},\boldsymbol{p})\cdot\gamma_t(\nabla\times\boldsymbol{E}(\boldsymbol{y};\boldsymbol{x},\boldsymbol{q}))+\gamma_t(\nabla\times\boldsymbol{E}(\boldsymbol{y};\boldsymbol{z},\boldsymbol{p}))\cdot\boldsymbol{E}(\boldsymbol{y};\boldsymbol{x},\boldsymbol{q})\mathrm{d}s(\boldsymbol{y}) \\
		&=\lim_{R\to\infty}\int_{\Gamma_R}-\boldsymbol{E}(\boldsymbol{y};\boldsymbol{z},\boldsymbol{p})\cdot\gamma_t(\nabla\times\boldsymbol{E}(\boldsymbol{y};\boldsymbol{x},\boldsymbol{q}))+\gamma_t(\nabla\times\boldsymbol{E}(\boldsymbol{y};\boldsymbol{z},\boldsymbol{p}))\cdot\boldsymbol{E}(\boldsymbol{y};\boldsymbol{x},\boldsymbol{q})\mathrm{d}s(\boldsymbol{y}) \\
        &=0.
	\end{aligned}
\end{equation}
It follows from a combination of (\ref{RR1})-(\ref{RR4}) that
\begin{align*}
	&\boldsymbol{q}\cdot\boldsymbol{E}(\boldsymbol{x};\boldsymbol{z},\boldsymbol{p})-\boldsymbol{p}\cdot\boldsymbol{E}(\boldsymbol{z};\boldsymbol{x},\boldsymbol{q}) \\
	=&\,\int_{\partial D}-\boldsymbol{E}(\boldsymbol{y};\boldsymbol{z},\boldsymbol{p})\cdot(\boldsymbol{\nu}\times\nabla_{\boldsymbol{y}}\times\mathbb{G}(\boldsymbol{x},\boldsymbol{y})\boldsymbol{q})+(\boldsymbol{\nu}\times\nabla\times\boldsymbol{E}(\boldsymbol{y};\boldsymbol{z},\boldsymbol{p}))\cdot\mathbb{G}(\boldsymbol{x},\boldsymbol{y})\boldsymbol{q}\mathrm{d}s(\boldsymbol{x}) \\
	&-\int_{\partial D}-\boldsymbol{E}(\boldsymbol{y};\boldsymbol{x},\boldsymbol{q})\cdot(\boldsymbol{\nu}\times\nabla_{\boldsymbol{y}}\times\mathbb{G}(\boldsymbol{z},\boldsymbol{y})\boldsymbol{p})+(\boldsymbol{\nu}\times\nabla\times\boldsymbol{E}(\boldsymbol{y};\boldsymbol{x},\boldsymbol{q}))\cdot\mathbb{G}(\boldsymbol{z},\boldsymbol{y})\boldsymbol{p}\mathrm{d}s(\boldsymbol{y}) \\
	&+\int_{\partial D}-\mathbb{G}(\boldsymbol{y},\boldsymbol{z})\boldsymbol{p}\cdot(\boldsymbol{\nu}\times\nabla_{\boldsymbol{y}}\times\mathbb{G}(\boldsymbol{x},\boldsymbol{y})\boldsymbol{q})+(\boldsymbol{\nu}\times\nabla_{\boldsymbol{y}}\times\mathbb{G}(\boldsymbol{y},\boldsymbol{z})\boldsymbol{p})\cdot\mathbb{G}(\boldsymbol{x},\boldsymbol{y})\boldsymbol{q}\mathrm{d}s(\boldsymbol{y}) \\
	&+\int_{\partial D}-\boldsymbol{E}(\boldsymbol{y};\boldsymbol{z},\boldsymbol{p})\cdot(\boldsymbol{\nu}\times\nabla\times\boldsymbol{E}(\boldsymbol{y};\boldsymbol{x},\boldsymbol{q}))+(\boldsymbol{\nu}\times\nabla\times\boldsymbol{E}(\boldsymbol{y};\boldsymbol{z},\boldsymbol{p}))\cdot\boldsymbol{E}(\boldsymbol{y};\boldsymbol{x},\boldsymbol{q})\mathrm{d}s(\boldsymbol{y}) \\
	=&\int_{\partial D}-(\mathbb{G}(\boldsymbol{y},\boldsymbol{z})\boldsymbol{p}+\boldsymbol{E}(\boldsymbol{y};\boldsymbol{z},\boldsymbol{p}))\cdot(\boldsymbol{\nu}\times\nabla\times(\mathbb{G}(\boldsymbol{x},\boldsymbol{y})\boldsymbol{q}+\boldsymbol{E}(\boldsymbol{y};\boldsymbol{x},\boldsymbol{q}))\mathrm{d}s(\boldsymbol{y}) \\
	&+\int_{\partial D}(\mathbb{G}(\boldsymbol{x},\boldsymbol{y})\boldsymbol{q}+\boldsymbol{E}(\boldsymbol{y};\boldsymbol{x},\boldsymbol{q}))\cdot\left(\boldsymbol{\nu}\times\nabla\times(\mathbb{G}(\boldsymbol{y},\boldsymbol{z})\boldsymbol{p}+\boldsymbol{E}(\boldsymbol{y};\boldsymbol{z},\boldsymbol{p}))\right)\mathrm{d}s(\boldsymbol{y}).
\end{align*}
Then from the impedance boundary condition
\begin{equation*}
	\boldsymbol{\nu}\times\nabla\times(\mathbb{G}(\boldsymbol{x},\boldsymbol{y})\boldsymbol{p}+\boldsymbol{E}(\boldsymbol{y};\boldsymbol{x},\boldsymbol{p}))-\mathrm{i}\eta(\mathbb{G}(\boldsymbol{x},\boldsymbol{y})\boldsymbol{p}+\boldsymbol{E}(\boldsymbol{y};\boldsymbol{x},\boldsymbol{p}))_T
	=\boldsymbol{0},
\end{equation*}
we get
\begin{equation*}
	\boldsymbol{q}\cdot\boldsymbol{E}(\boldsymbol{x};\boldsymbol{z},\boldsymbol{p})
	=\boldsymbol{p}\cdot\boldsymbol{E}(\boldsymbol{z};\boldsymbol{x},\boldsymbol{q}),
\end{equation*}
which completes the proof.
\end{proof}

\begin{theorem}
Let $ D_1 $ and $ D_2 $ be two scatterers. For a fixed wavenumber $ k $, if the tangential components of the electric field $ \boldsymbol{E}_1(\cdot;\boldsymbol{z},\boldsymbol{p}) $ and $ \boldsymbol{E}_2(\cdot;\boldsymbol{z},\boldsymbol{p}) $ on $\Gamma_R$, resulting from the scattering of the point source $\boldsymbol{E}^\mathrm{inc}(\cdot;\boldsymbol{z},\boldsymbol{p})=\mathbb{G}(\cdot,\boldsymbol{z})\boldsymbol{p}$ by $ D_1 $ and $ D_2 $, respectively, coninside for all $ \boldsymbol{z}\in\Gamma_{R^\prime} $ and polarizations $ \boldsymbol{p} $, then $ D_1=D_2 $.
\end{theorem}
\begin{proof}
Denote $ G=\mathbb{R}_Z^3\backslash\overline{D_1 \cup D_2} $. Suppose that $ \boldsymbol{u}(\boldsymbol{x};\boldsymbol{z},\boldsymbol{p})=\boldsymbol{E}_1(\boldsymbol{x};\boldsymbol{z},\boldsymbol{p})-\boldsymbol{E}_2(\boldsymbol{x};\boldsymbol{z},\boldsymbol{p})$ has the null tangential component $ \boldsymbol{n}\times\boldsymbol{u}(\boldsymbol{x};\boldsymbol{z},\boldsymbol{p})=\mathbf{0} $ on $ \Gamma_R $. Then $ \boldsymbol{u} $ satisfies the boundary value problem
	\begin{equation*}
		\left\{
		\begin{aligned}
			&\nabla\times\nabla\times\boldsymbol{u}-k^2\boldsymbol{u}=\mathbf{0} \quad &&\text{in}~\mathbb{R}_Z^3\backslash\overline{C_R}, \\
			&\boldsymbol{\nu}\times\boldsymbol{u}=\mathbf{0} \quad &&\text{on}~\Gamma^\pm\backslash\Gamma_R^\pm, \\
			&\boldsymbol{n}\times\boldsymbol{u}=\mathbf{0} \quad &&\text{on}~\Gamma_R.
		\end{aligned}
		\right.
	\end{equation*}
By the series expansions (\ref{series expansion of Er})-(\ref{series expansion of Ez}) and the divergence-free condition, we obtain that
\begin{align*}
	A_{nm}^+H_{n+1}^{(1)}(k_mR)-A_{nm}^-H_{n-1}^{(1)}(k_mR)
	&=0, \\
	C_{nm}H_n^{(1)}(k_mR)
	&=0, \\
	\frac{k_m}{2}A_{nm}^+H_n^{(1)}(k_mR)-\frac{k_m}{2}A_{nm}^-H_n^{(1)}(k_mR)-\frac{m\pi}{Z}C_{nm}H_n^{(1)}(k_mR)
	&=0.
\end{align*}
These combined with Lemma \ref{le 3.1} yield $ \boldsymbol{u}(\boldsymbol{x};\boldsymbol{z},\boldsymbol{p})=\mathbf{0} $ in $ \mathbb{R}_Z^3\backslash\overline{C_R} $. Then, by the unique continuation principle and Lemma \ref{Reciprocity relation}, it can be concluded that $ \boldsymbol{u}(\boldsymbol{x};\boldsymbol{z},\boldsymbol{p})=\mathbf{0} $ for all $ \boldsymbol{x}\in\Gamma_{R^\prime} $ and $ \boldsymbol{z} \in G $.

Now assume that $ D_1 \neq D_2 $. Without loss of generality, there exists some $ \boldsymbol{x}^*\in\partial G $ such that $ \boldsymbol{x}^*\in\partial D_1 $ and $ \boldsymbol{x}^*\notin\overline{D_2} $. In particular, let
\begin{equation*}
	\boldsymbol{z}_n
	:=\boldsymbol{x}^*+\frac{1}{n}\boldsymbol{n}(\boldsymbol{x}^*) \in G, \quad n=1,2,\dots,
\end{equation*}
for sufficiently large $ n $. Then, on one hand we obtain that
\begin{equation*}
	\lim_{n\to\infty}\mathscr{B}[\boldsymbol{E}_2(\boldsymbol{x}^*;\boldsymbol{z}_n,\boldsymbol{p})]=\mathscr{B}[\boldsymbol{E}_2(\boldsymbol{x}^*;\boldsymbol{x}^*,\boldsymbol{p})],
\end{equation*}
On the other hand, we find that
\begin{equation*}
	\lim_{n\to\infty}\mathscr{B}[\boldsymbol{E}_1(\boldsymbol{x}^*;\boldsymbol{z}_n,\boldsymbol{p})]=\infty,
\end{equation*}
because of the boundary condition $ \mathscr{B}[\boldsymbol{E}_1(\boldsymbol{x}^*;\boldsymbol{z}_n,\boldsymbol{p})]=-\mathscr{B}[\mathbb{G}(\boldsymbol{x}^*,\boldsymbol{z}_n)\boldsymbol{p}] $. This contradicts that $ \boldsymbol{E}_1(\boldsymbol{x}^*;\boldsymbol{z}_n,\boldsymbol{p})=\boldsymbol{E}_2(\boldsymbol{x}^*;\boldsymbol{z}_n,\boldsymbol{p}) $ for all sufficiently large $ n $, and therefore $ D_1=D_2 $.
\end{proof}

\section{Conclusion}\label{sec 5}
This paper studied both the direct and inverse problems of time-harmonic electromagnetic scattering by an impedance obstacle in a PEC parallel-plate waveguide. The TBC defined by a Calder\'on operator is constructed to reduce the scattering problem into a truncated domain. With the help of properties of the Calder\'on operator, the well-posedness of direct problem is obtained by applying the variational method. Based on the probe method, the uniqueness result of inverse problem follows by deriving the Green's representation formula and reciprocity relation.
Some future extensions to the numerical analysis of the electromagnetic waveguide scattering problems as well analysis and numerical schemes for inverse scattering problems will be considered in the future.

\section*{Acknowledgment}
This work is partially supported by the National Key Research and Development Program of China (No. 2024YFA1016000), the Strategic Priority Research Program of the Chinese Academy of Sciences through Grant No. XDB0640000 and NSFC through Grants No. 12271082, 62231016, 12171465.

\appendix
\section{Proof of Lemma~\ref{lem-dual}}
\label{appendixA}
This appendix presents the proof of Lemma~\ref{lem-dual}. Denote $ \alpha_n=\mathrm{i}n/R $ and $ \beta_m=m\pi/Z $. We can verify that
    \begin{equation}\label{equivalent L2t on GammaR}
        \begin{aligned}
        \langle\bm{u},\bm{v}\rangle_{\Gamma_R}=&~\pi RZ\sum_{m=0}^{\infty}\sum_{n=-\infty}^{\infty}(1+|\alpha_n|^2+\beta_m^2)^{-1}\times \\
            &\left[\left(\alpha_nu_{\theta,s,nm}-(\beta_m-\mathrm{i})u_{z,c,nm}\right)\overline{\left(\alpha_nv_{\theta,s,nm}-(\beta_m-\mathrm{i})v_{z,c,nm}\right)}\right. \\
&\quad+\left(\alpha_nu_{z,c,nm}-(\beta_m+\mathrm{i})u_{\theta,s,nm}\right)\overline{\left(\alpha_nv_{z,c,nm}-(\beta_m+\mathrm{i})v_{\theta,s,nm}\right)} \\
&\quad+\left(\alpha_nu_{\theta,c,nm}+(\beta_m+\mathrm{i})u_{z,s,nm}\right)\overline{\left(\alpha_nv_{\theta,c,nm}+(\beta_m+\mathrm{i})v_{z,s,nm}\right)} \\
&\quad+\left.\left(\alpha_nu_{z,s,nm}+(\beta_m-\mathrm{i})u_{\theta,c,nm}\right)\overline{\left(\alpha_nv_{z,s,nm}+(\beta_m-\mathrm{i})v_{\theta,c,nm}\right)}\right]
        \end{aligned}
    \end{equation}
    and
    \begin{equation}
        |u_{\theta,s,nm}|^2+|u_{z,c,nm}|^2
        =(1+|\alpha_n|^2+\beta_m^2)^{-1}(|\alpha_nu_{\theta,s,nm}-(\beta_m-\mathrm{i})u_{z,c,nm}|^2+|\alpha_nu_{z,c,nm}-(\beta_m+\mathrm{i})u_{\theta,s,nm}|^2),
    \end{equation}
    \begin{equation}
        |u_{\theta,c,nm}|^2+|u_{z,s,nm}|^2
        =(1+|\alpha_n|^2+\beta_m^2)^{-1}(|\alpha_nu_{\theta,c,nm}+(\beta_m+\mathrm{i})u_{z,s,nm}|^2+|\alpha_nu_{z,s,nm}+(\beta_m-\mathrm{i})u_{\theta,c,nm}|^2).
    \end{equation}
    Note that
    \begin{align*}
        &\frac{1}{3}|\alpha_nu_{\theta,s,nm}-(\beta_m-\mathrm{i})u_{z,c,nm}|^2-\frac{1}{2}|u_{z,c,nm}|^2 \\
        &\leq|\alpha_nu_{\theta,s,nm}-\beta_mu_{z,c,nm}|^2 \\
        &\leq|\alpha_nu_{\theta,s,nm}-(\beta_m-\mathrm{i})u_{z,c,nm}|^2+|u_{z,c,nm}|^2,
    \end{align*}
    and some analogous inequalities hold for other terms in (\ref{equivalent L2t on GammaR}). We can obtain the equivalent norms on $ \bm{H}^{-1/2}(\mathrm{Div},\Gamma_R) $ and $ \bm{H}^{-1/2}(\mathrm{Curl},\Gamma_R) $ as
    \begin{equation}\label{equivalent norm for HDiv}
        \begin{aligned}
            \|\bm{u}\|_{\bm{H}^{-1/2}(\mathrm{Div},\Gamma_R)}^2
            \simeq\sum_{m=0}^{\infty}\sum_{n=-\infty}^{\infty}&\left[(1+|\alpha_n|^2+\beta_m^2)^{-1/2}|\alpha_nu_{\theta,s,nm}-(\beta_m-\mathrm{i})u_{z,c,nm}|^2\right. \\
            &+(1+|\alpha_n|^2+\beta_m^2)^{-3/2}|\alpha_nu_{z,c,nm}-(\beta_m+\mathrm{i})u_{\theta,s,nm}|^2 \\
            &+(1+|\alpha_n|^2+\beta_m^2)^{-1/2}|\alpha_nu_{\theta,c,nm}+(\beta_m+\mathrm{i})u_{z,s,nm}|^2 \\
            &+\left.(1+|\alpha_n|^2+\beta_m^2)^{-3/2}|\alpha_nu_{z,s,nm}+(\beta_m-\mathrm{i})u_{\theta,c,nm}|^2\right]
        \end{aligned}
    \end{equation}
    and
    \begin{equation}\label{equivalent norm for HCurl}
        \begin{aligned}
            \|\bm{u}\|_{\bm{H}^{-1/2}(\mathrm{Curl},\Gamma_R)}^2
            \simeq\sum_{m=0}^{\infty}\sum_{n=-\infty}^{\infty}&\left[(1+|\alpha_n|^2+\beta_m^2)^{-3/2}|\alpha_nu_{\theta,s,nm}-(\beta_m-\mathrm{i})u_{z,c,nm}|^2\right. \\
            &+(1+|\alpha_n|^2+\beta_m^2)^{-1/2}|\alpha_nu_{z,c,nm}-(\beta_m+\mathrm{i})u_{\theta,s,nm}|^2 \\
            &+(1+|\alpha_n|^2+\beta_m^2)^{-3/2}|\alpha_nu_{\theta,c,nm}+(\beta_m+\mathrm{i})u_{z,s,nm}|^2 \\
            &+\left.(1+|\alpha_n|^2+\beta_m^2)^{-1/2}|\alpha_nu_{z,s,nm}+(\beta_m-\mathrm{i})u_{\theta,c,nm}|^2\right],
        \end{aligned}
    \end{equation}
    where $ \simeq $ stands for the equivalence. It follows from (\ref{equivalent L2t on GammaR})–(\ref{equivalent norm for HCurl}) and the Cauchy–Schwarz inequality that there exists a constant $ C>0 $ such that
	\begin{equation*}
		\langle\bm{u},\bm{v}\rangle_{\Gamma_R}
		\leq C\|\bm{u}\|_{\bm{H}^{-1/2}(\mathrm{Div},\Gamma_R)}\|\bm{v}\|_{\bm{H}^{-1/2}(\mathrm{Curl},\Gamma_R)}.
	\end{equation*}
	This implies that $ \langle\bm{u},\bm{v}\rangle_{\Gamma_R} $ is well defined for any $ \bm{u}\in\bm{H}^{-1/2}(\mathrm{Div},\Gamma_R) $ and $ \bm{v}\in\bm{H}^{-1/2}(\mathrm{Curl},\Gamma_R) $.

    Define two operators $ \mathscr{C}:\bm{H}^{-1/2}(\mathrm{Div},\Gamma_R)\to\bm{H}^{-1/2}(\mathrm{Curl},\Gamma_R) $ and $ \mathscr{D}:\bm{H}^{-1/2}(\mathrm{Curl},\Gamma_R)\to\bm{H}^{-1/2}(\mathrm{Div},\Gamma_R) $ by
	\begin{align*}
		(\mathscr{C}[\bm{u}])_{\theta,s,nm}
		=&~\left[(\beta_m^2+1)(1+|\alpha_n|^2+\beta_m^2)^{-3/2}-\alpha_n^2(1+|\alpha_n|^2+\beta_m^2)^{-1/2}\right]u_{\theta,s,nm} \\
		&+\alpha_n(\beta_m-\mathrm{i})\left[(1+|\alpha_n|^2+\beta_m^2)^{-1/2}-(1+|\alpha_n|^2+\beta_m^2)^{-3/2}\right]u_{z,c,nm}, \\
		(\mathscr{C}[\bm{u}])_{z,c,nm}
		=&~\alpha_n(\beta_m+\mathrm{i})\left[(1+|\alpha_n|^2+\beta_m^2)^{-3/2}-(1+|\alpha_n|^2+\beta_m^2)^{-1/2}\right]u_{\theta,s,nm} \\
		&+\left[(\beta_m^2+1)(1+|\alpha_n|^2+\beta_m^2)^{-1/2}-\alpha_n^2(1+|\alpha_n|^2+\beta_m^2)^{-3/2}\right]u_{z,c,nm}, \\
		(\mathscr{C}[\bm{u}])_{\theta,c,nm}
		=&~\left[(\beta_m^2+1)(1+|\alpha_n|^2+\beta_m^2)^{-3/2}-\alpha_n^2(1+|\alpha_n|^2+\beta_m^2)^{-1/2}\right]u_{\theta,c,nm} \\
		&+\alpha_n(\beta_m+\mathrm{i})\left[(1+|\alpha_n|^2+\beta_m^2)^{-3/2}-(1+|\alpha_n|^2+\beta_m^2)^{-1/2}\right]u_{z,s,nm}, \\
		(\mathscr{C}[\bm{u}])_{z,s,nm}
		=&~\alpha_n(\beta_m-\mathrm{i})\left[(1+|\alpha_n|^2+\beta_m^2)^{-1/2}-(1+|\alpha_n|^2+\beta_m^2)^{-3/2}\right]u_{\theta,c,nm} \\
		&+\left[(\beta_m^2+1)(1+|\alpha_n|^2+\beta_m^2)^{-1/2}-\alpha_n^2(1+|\alpha_n|^2+\beta_m^2)^{-3/2}\right]u_{z,s,nm}
	\end{align*}
	and
	\begin{align*}
		(\mathscr{D}[\bm{v}])_{\theta,s,nm}
		=&~\left[(\beta_m^2+1)(1+|\alpha_n|^2+\beta_m^2)^{-1/2}-\alpha_n^2(1+|\alpha_n|^2+\beta_m^2)^{-3/2}\right]v_{\theta,s,nm} \\
		&+\alpha_n(\beta_m-\mathrm{i})\left[(1+|\alpha_n|^2+\beta_m^2)^{-3/2}-(1+|\alpha_n|^2+\beta_m^2)^{-1/2}\right]v_{z,c,nm}, \\
		(\mathscr{D}[\bm{v}])_{z,c,nm}
		=&~\alpha_n(\beta_m+\mathrm{i})\left[(1+|\alpha_n|^2+\beta_m^2)^{-1/2}-(1+|\alpha_n|^2+\beta_m^2)^{-3/2}\right]v_{\theta,s,nm} \\
		&+\left[(\beta_m^2+1)(1+|\alpha_n|^2+\beta_m^2)^{-3/2}-\alpha_n^2(1+|\alpha_n|^2+\beta_m^2)^{-1/2}\right]v_{z,c,nm}, \\
		(\mathscr{D}[\bm{v}])_{\theta,c,nm}
		=&\left[(\beta_m^2+1)(1+|\alpha_n|^2+\beta_m^2)^{-1/2}-\alpha_n^2(1+|\alpha_n|^2+\beta_m^2)^{-3/2}\right]v_{\theta,c,nm} \\
		&+\alpha_n(\beta_m+\mathrm{i})\left[(1+|\alpha_n|^2+\beta_m^2)^{-1/2}-(1+|\alpha_n|^2+\beta_m^2)^{-3/2}\right]v_{z,s,nm}, \\
		(\mathscr{D}[\bm{v}])_{z,s,nm}
		=&~\alpha_n(\beta_m-\mathrm{i})\left[(1+|\alpha_n|^2+\beta_m^2)^{-3/2}-(1+|\alpha_n|^2+\beta_m^2)^{-1/2}\right]v_{\theta,c,nm} \\
		&+\left[(\beta_m^2+1)(1+|\alpha_n|^2+\beta_m^2)^{-3/2}-\alpha_n^2(1+|\alpha_n|^2+\beta_m^2)^{-1/2}\right]v_{z,s,nm}.
	\end{align*}
    It is clear to show that
	\begin{align*}
		\alpha_n(\mathscr{C}[\boldsymbol{u}])_{\theta,s,nm}-(\beta_m-\mathrm{i})(\mathscr{C}[\boldsymbol{u}])_{z,c,nm}
		&=(1+|\alpha_n|^2+\beta_m^2)^{1/2}\left(\alpha_nu_{\theta,s,nm}-(\beta_m-\mathrm{i})u_{z,c,nm}\right), \\
		\alpha_n(\mathscr{C}[\boldsymbol{u}])_{z,c,nm}-(\beta_m+\mathrm{i})(\mathscr{C}[\boldsymbol{u}])_{\theta,s,nm}
		&=(1+|\alpha_n|^2+\beta_m^2)^{-1/2}\left(\alpha_nu_{z,c,nm}-(\beta_m+\mathrm{i})u_{\theta,s,nm}\right), \\
		\alpha_n(\mathscr{C}[\boldsymbol{u}])_{\theta,c,nm}+(\beta_m+\mathrm{i})(\mathscr{C}[\boldsymbol{u}])_{z,s,nm}
		&=(1+|\alpha_n|^2+\beta_m^2)^{1/2}\left(\alpha_nu_{\theta,c,nm}+(\beta_m+\mathrm{i})u_{z,s,nm}\right), \\
		\alpha_n(\mathscr{C}[\boldsymbol{u}])_{z,s,nm}+(\beta_m-\mathrm{i})(\mathscr{C}[\boldsymbol{u}])_{\theta,c,nm}
		&=(1+|\alpha_n|^2+\beta_m^2)^{-1/2}\left(\alpha_nu_{z,s,nm}+(\beta_m-\mathrm{i})u_{\theta,c,nm}\right), \\
		\alpha_n(\mathscr{D}[\boldsymbol{v}])_{\theta,s,nm}-(\beta_m-\mathrm{i})(\mathscr{D}[\boldsymbol{v}])_{z,c,nm}
		&=(1+|\alpha_n|^2+\beta_m^2)^{-1/2}\left(\alpha_nv_{\theta,s,nm}-(\beta_m-\mathrm{i})v_{z,c,nm}\right), \\
		\alpha_n(\mathscr{D}[\boldsymbol{v}])_{z,c,nm}-(\beta_m+\mathrm{i})(\mathscr{D}[\boldsymbol{v}])_{\theta,s,nm}
		&=(1+|\alpha_n|^2+\beta_m^2)^{1/2}\left(\alpha_nv_{z,c,nm}-(\beta_m+\mathrm{i})v_{\theta,s,nm}\right), \\
		\alpha_n(\mathscr{D}[\boldsymbol{v}])_{\theta,c,nm}+(\beta_m+\mathrm{i})(\mathscr{D}[\boldsymbol{v}])_{z,s,nm}
		&=(1+|\alpha_n|^2+\beta_m^2)^{-1/2}\left(\alpha_nv_{\theta,c,nm}+(\beta_m+\mathrm{i})v_{z,s,nm}\right), \\
		\alpha_n(\mathscr{D}[\boldsymbol{v}])_{z,s,nm}+(\beta_m-\mathrm{i})(\mathscr{D}[\boldsymbol{v}])_{\theta,c,nm}
		&=(1+|\alpha_n|^2+\beta_m^2)^{1/2}\left(\alpha_nv_{z,s,nm}+(\beta_m-\mathrm{i})v_{\theta,c,nm}\right).
	\end{align*}
	We can obtain from (\ref{equivalent L2t on GammaR})-(\ref{equivalent norm for HCurl}) that
	\begin{equation*}
		\langle\mathscr{D}[\bm{v}],\mathscr{C}[\bm{u}]\rangle_{\Gamma_R}
		=\overline{\langle\bm{u},\bm{v}\rangle_{\Gamma_R}}
	\end{equation*}
	and
	\begin{align*}
		\|\mathscr{C}[\bm{u}]\|_{\bm{H}^{-1/2}(\mathrm{Curl},\Gamma_R)}
		&\simeq\|\bm{u}\|_{\bm{H}^{-1/2}(\mathrm{Div},\Gamma_R)}, \\
		\|\mathscr{D}[\bm{v}]\|_{\bm{H}^{-1/2}(\mathrm{Div},\Gamma_R)}
		&\simeq\|\bm{v}\|_{\bm{H}^{-1/2}(\mathrm{Curl},\Gamma_R)}.
	\end{align*}
which completes the proof.

\bibliographystyle{abbrv}
\bibliography{ref}

\end{document}